\documentclass[a4paper,reqno,11pt]{amsart}
\usepackage{amsmath,amssymb,amsbsy,amsfonts,amsthm,latexsym,amsopn,amstext,amsxtra,euscript,amscd,amsthm,enumerate}
\usepackage[vmargin=2cm,hmargin=2cm]{geometry}   
\usepackage[xcolor]{changebar}
\cbcolor{red}

\newtheorem{lem}{Lemma}
\newtheorem{lemma}[lem]{Lemma}

\newtheorem{prop}{Proposition}

\newtheorem{thm}{Theorem}
\newtheorem{theorem}[thm]{Theorem}

\newtheorem{cor}{Corollary}
\newtheorem{corollary}[cor]{Corollary}

\theoremstyle{definition}
\newtheorem{defi}{Definition}
\newtheorem{definition}[defi]{Definition}

\newtheorem{exam}{Example}

\newtheorem{rem}{Remark}

\def\\{\cr}
\def\({\left(}
\def\){\right)}
\def\[{\left[}
\def\]{\right]}
\def\<{\langle}
\def\>{\rangle}

\def\cA{{\mathcal A}}
\def\cB{{\mathcal B}}

\def\cD{{\mathcal D}}

\def\cF{{\mathcal F}}

\def\cM{{\mathcal M}}

\def\cS{{\mathcal S}}
\def\cP{{\mathcal P}}

\def\Q{{\mathbb Q}}
\def\R{{\mathbb R}}

\DeclareMathOperator{\lcm}{lcm}

\makeatletter
\let\@@pmod\pmod
\DeclareRobustCommand{\pmod}{\@ifstar\@pmods\@@pmod}
\def\@pmods#1{\mkern4mu({\operator@font mod}\mkern 6mu#1)}
\makeatother

\begin{document}

\title{On the discriminator of
Lucas sequences. II}

\author{{\sc M.\,Ferrari,}~~{\sc F.\,Luca}~~{\sc and P.\,Moree}
}

\address{
{Politecnico di Torino}\newline
{Corso Duca degli Abruzzi 24, 10129, Torino, Italy}}
\email{\tt matteo.ferrari@polito.it}

\address{{School of Mathematics, University of the Witwatersrand}\newline
{P.\,O.\,Box Wits 2050, South Africa}}
\email{\tt florian.luca@wits.ac.za}

\address{
{Max Planck Institute for Mathematics}\newline
{Vivatsgasse 7, 53111 Bonn, Germany}}
\email{\tt moree@mpim-bonn.mpg.de}

\date{\today}

\pagenumbering{arabic}

\begin{abstract}
\noindent 
The family of Shallit sequences consists of
the Lucas sequences
satisfying the recurrence $U_{n+2}(k)=(4k+2)U_{n+1}(k) -U_n(k),$ with
initial values $U_0(k)=0$
and $U_1(k)=1$ and with $k\ge 1$ arbitrary. 
For every fixed $k$ the integers $\{U_n(k)\}_{n\ge 0}$ are
distinct, and hence for every $n\ge 1$ 
there exists a smallest integer $\cD_k(n)$, 
called \emph{discriminator}, such that
$U_0(k),U_1(k),\ldots,U_{n-1}(k)$ are pairwise incongruent modulo 
$\cD_k(n).$
In part I  it was proved that there exists a constant $n_k$ such that $\cD_{k}(n)$
has a simple characterization for every $n\ge n_k$. Here, 
we study the values not following
this characterization and provide an upper bound for 
$n_k$ using Matveev's 
theorem and the Koksma-Erd\H os-Tur\'an
inequality.
We completely
determine  the
discriminator $\cD_{k}(n)$ for every
$n\ge 1$ and a set of integers $k$ of natural density $68/75$.  
We also correct an omission in
the statement of 
Theorem 3 in part I.
\end{abstract}

\maketitle

\section{Introduction}
\label{intro}
\subsection{Motivation}The {\it discriminator} of a sequence ${\bf a}=\{a_n\}_{n\ge 0}$
of distinct integers is the sequence $\{\cD_{\bf a}(n)\}_{n\ge 0}$
with
$$\cD_{\bf a}(n):=\min\{m\ge 1: a_0,\ldots,a_{n-1}~{\text{\rm
are~pairwise~distinct~modulo}}~m\}.
$$
In other words,
$\cD_{\bf a}(n)$
is the smallest positive
integer $m$ that discriminates (tells
apart) the integers $a_0,\ldots, a_{n-1}$  on reducing them modulo $m$.

Note that 
$n\le \cD_{\bf a}(n)\le \max\{a_0,\ldots,a_{n-1}\}
-\min\{a_0,\ldots,a_{n-1}\}+1.$
Put
$$\cD_{\bf a}:=\{\cD_{\bf a}(n): n\ge 1\}.
$$
The main challenge is to give an easy description or characterization of
$\cD_{\bf a}(n)$. 
Unfortunately for most sequences
${\bf a}$ such a characterization does not seem to exist 
(see part I by Faye, Luca and
Moree \cite{FLM} for
references to the earlier
literature).

Let  $\textbf{U}(k)$ be the 
sequence $\{U_n(k)\}_{n\ge 0}$ with $U_0(k)=0,~U_1(k)=1$ and
$$
U_{n+2}(k)=(4k+2)U_{n+1}(k)-U_n(k).
$$
It has Binet form as in \eqref{Binetform} below.
In this paper we continue the work 
of determining the discriminator
$\cD_{{\bf U}(k)}(n)$ 
(which for brevity we denote by $\cD_{k}(n)$), which
was initiated in part I. A typical example is provided in Tab.\,\ref{tab:16}.
The project is inspired by conjectures made by Jeffrey Shallit.
\begin{table}[h]
\centering
\begin{tabular}{|c|c|c|c|c|c|c|c|}\hline
$n$ & $D_{16}(n)$ & $n$ & $D_{16}(n)$ & $n$ & $D_{16}$ & $n$ & $D_{16}(n)$ \\
\hline \hline
$1$ & $1$ & $257 - 272$ & $2^{4} \cdot 17$ & $2313 - 2400$ & $2^{5} \cdot 5^3$ & $19653 - 32768$ & $2^{15}$ \\
\hline
$2$ & $2$ & $273 - 300$ & $2^{2}\cdot 5^3$ & $2401 - 4096$ & $2^{12}$ & $32769 - 34816$ & $2^{11} \cdot 17$ \\
\hline
$3 - 4$ & $2^{2}$ & $301 - 512$ & $2^{9}$ & $4097 - 4352$ & $2^{8} \cdot 17$ & $34817 - 36992$ & $2^{7} \cdot 17^{2}$ \\
\hline
$5 - 8$ & $2^{3}$ & $513 - 544$ & $2^{5} \cdot 17$ & $4353 - 4624$ & $2^{4} \cdot 17^{2}$ & $36993 - 39304$ & $2^{3} \cdot 17^{3}$ \\
\hline
$9 - 16$ & $2^{4}$ & $545 - 578$ & $2 \cdot 17^{2}$ & $4625 - 4800$ & $2^{6}\cdot 5^3$ & $39305 - 65536$ & $2^{16}$ \\
\hline
$17 - 32$ & $2^{5}$ & $579 - 600$ & $2^{3}\cdot 5^3$ & $4801 - 8192$ & $2^{13}$ & $65537 - 69632$ & $2^{12} \cdot 17$ \\
\hline
$33 - 34$ & $2 \cdot 17$ & $601 - 1024$ & $2^{10}$ & $8193 - 8704$ & $2^{9} \cdot 17$ & $69633 - 73984$ & $2^{8} \cdot 17^{2}$ \\
\hline
$35 - 64$ & $2^{6}$ & $1025 - 1088$ & $2^{6} \cdot 17$ & $8705 - 9248$ & $2^{5} \cdot 17^{2}$ & $73985 - 78608$ & $2^{4} \cdot 17^{3}$ \\
\hline
$65 - 68$ & $2^{2} \cdot 17$ & $1089 - 1156$ & $2^{2} \cdot 17^{2}$ & $9249 - 9826$ & $2 \cdot 17^{3}$ & $78609 - 131072$ & $2^{17}$ \\
\hline
$69 - 128$ & $2^{7}$ & $1157 - 1200$ & $2^{4} \cdot 5^3$ & $9827 - 16384$ & $2^{14}$ & $131073 - 139264$ & $2^{13} \cdot 17$ \\
\hline
$129 - 136$ & $2^{3} \cdot 17$ & $1201 - 2048$ & $2^{11}$ & $16385 - 17408$ & $2^{10} \cdot 17$ & $139265 - 147968$ & $2^{9} \cdot 17^{2}$ \\
\hline
$137 - 150$ & $2\cdot 5^3$ & $2049 - 2176$ & $2^{7} \cdot 17$ & $17409 - 18496$ & $2^{6} \cdot 17^{2}$ & $147969 - 157216$ & $2^{5} \cdot 17^{3}$ \\
\hline
$151 - 256$ & $2^{8}$ & $2177 - 2312$ & $2^{3} \cdot 17^{2}$ & $18497 - 19652$ & $2^{2} \cdot 17^{3}$ & $157217 - 167042$ & $2 \cdot 17^{4}$ \\
\hline
\end{tabular}
\medskip \medskip
\caption{The discriminator for $k=16$ and $1\le n\le 167042$}
\label{tab:16}
\end{table}

We now recall the two main results from part I.
\begin{theorem}
\label{1+2}~
\begin{enumerate}[{\rm a)}]
\item Let $s_n$ be the smallest power of $2$
such that
$s_n\ge n$. Let $t_n$ be the smallest integer
of the form $2^a\cdot 5^b$ satisfying $2^a\cdot 5^b\ge 5n/3$
with $a,b\ge 1$.
Then $$\cD_1(n)=\min\{s_n,t_n\}.$$
\item Let $e\ge 0$ be the smallest integer such that
$2^e\ge n$ and $f\ge 1$ the smallest  integer such that
$3\cdot 2^f\ge n$. Then
$$\cD_2(n)=\min\{2^e,3\cdot 2^f\}.$$
\end{enumerate}
\end{theorem}

\indent The second main result shows
that the behavior of the discriminator 
$\cD_k$ 
with $k>2$ is very different from that of
$\cD_1$. 
It corrects Theorem 3 in part I, where
the conditions on $k$ involving $6\pmod*{9}$ in
the definition of $\cB_{k}$ were 
erroneously omitted. For details see 
Sec.\,\ref{proofcorrected}.
\begin{theorem}[Corrected version of Theorem 3 in 
\cite{FLM}]
\label{oldmaincorrected}
Put
$$\cA_k=\begin{cases}
\{m~{\text{\rm odd}}:\text{\rm if~}p\mid m,~{\text{\rm then}}~p\mid k\} & \text{~if~}
k\not\equiv 6\pmod*{9};\cr
\{m~{\text{\rm odd}},~9\nmid m:\text{\rm if~}p\mid m,~{\text{\rm then}}~p\mid k\} &\text{~if~}
k\equiv 6\pmod*{9},
\end{cases}
$$
and $$\cB_k
=\begin{cases}
\{m~{\text{\rm even}}:\text{\rm if~}p\mid m,~{\text{\rm then}}~p\mid k(k+1)\}
& \text{~if~} k \not\equiv 6 \pmod*{9} \text{~and~}
k\not\equiv 2\pmod*{9};\cr
\{m~{\text{\rm even},~9\nmid m}:\text{\rm if~}p\mid m,~{\text{\rm then}}~p\mid k(k+1)\}
& \text{~if~}
k \equiv 6 \pmod*{9}
\text{~or~}
k\equiv 2 \pmod*{9}.
\end{cases}
$$
Let $k>2$ be fixed.
We have 
\begin{equation}
\label{Dkn=n}
\cD_{k}(n)=n\iff n\in \cA_{k}\cup \cB_{k}.
\end{equation}
Furthermore, 
\begin{equation}
\label{dkinequal2}
\cD_{k}(n)\le \min\{m\ge n:
m\in \cA_{k}\cup \cB_{k}\},
\end{equation}
with equality if the interval $[n,3n/2)$ contains an integer 
$m\in \cA_{k}\cup \cB_{k}$.
There are at most 
finitely many $n$ for which in \eqref{dkinequal2} strict inequality holds.
\end{theorem}
\begin{corollary}\label{cor:twopower}
Let $k\ge 1$. Then 
$\cD_k(n)\le \min\{2^b:2^b\ge n\}$.
\end{corollary}
\noindent The corollary is rather trivial and can be easily proved 
directly (\cite[Lemma 1]{FLM}).
\par In \eqref{dkinequal2} equality is more likely to hold if the
set $\cA_k\cup \cB_k$ contains many elements. This 
will happen if $k(k+1)$ is divisible by
a small odd prime. Related to this the following quantity will play a role.
\begin{definition}
\label{def:xpalpha}
Let $\alpha>1$ be a given
real number and
$p\ge 3$ be an arbitrary prime. 
We denote by $n_p(\alpha)$ 
the smallest integer $m$
such
that the interval $[n,n\alpha)$ contains
an \emph{even} integer of the form $2^a\cdot p^b$ for every integer $n\ge m$. If we drop the evenness requirement, we will write $n_p^o(\alpha)$.
\end{definition}
\noindent The existence of $n_p(\alpha)$ is guaranteed by \cite[Lemma 19]{FLM}. 
Note that $n_p^o(\alpha)\le n_p(\alpha)$. 
If $n_1,n_2,\ldots$ is an infinite sequence of integers of the required form with $1<\frac{n_{j+1}}{n_j}<\alpha$ for every $j\ge 1$, then it is
easy to see that $n_p^o(\alpha)\le n_1$ (for further details see 
Sec.\,\ref{sec:interval}). 
\par We do not need more than $n_p^o(\alpha)$ and $n_p(\alpha)$, but the same ideas apply to numbers of the form $p^a\cdot q^b$, with $p$
and $q$ distinct primes. If $m_1,m_2,\ldots$ is the ordered sequence of these
numbers, then it was shown by Tijdeman 
\cite{Tijdeman} that there exist effectively computable constants $c_1$ and $c_2$ such that
$(\log m_j)^{-c_1}\ll \frac{m_{j+1}}{m_j}-1 \ll (\log m_j)^{-c_2}$. Later 
Langevin \cite{Langevin} gave
explicit values for $c_1$ and $c_2$,
which were recently improved by 
Languasco et al.\,\cite{LLMT}.
\par Theorem \ref{main2} is our main result and gives a
complete characterization of
$\cD_k(n)$ for every $n\ge 1$ and 
\begin{equation}
\label{excludedcongruenceclasses}    
k\not \equiv 2,6,7,12,17,18,22\pmod*{25},
\end{equation}
where for brevity we write $k\not\equiv a,b\pmod*{m}$ for
$k\not\equiv a\pmod*{m}$ and $k\not\equiv b\pmod*{m}$,
and $k\equiv a,b\pmod*{m}$ for
$k\equiv a\pmod*{m}$ or $k\equiv b\pmod*{m}$.
It sharpens Theorem \ref{oldmaincorrected} and will
be proved in Sect.\,\ref{sec:mainproof}. The proof strategy is discussed in Sect.\,\ref{proofstrategy}. 
\begin{theorem}
\label{main2}
Let $k\ge 1$ be an arbitrary integer.
Let $\cA_k$ and $\cB_k$ be as in 
Theorem \ref{oldmaincorrected} 
and define  $$\cS_{k,n}:= \{m \in \cA_k \cup \cB_k : m \ge n\}.$$
Then $$\cD_k(n) = \min 
\cS_{k,n} \text{~~~if~~~}k\not\equiv 1\pmod*{3}.$$
Next suppose that 
$k\equiv 1\pmod*{3},$ $k\not\equiv 2\pmod*{5}$ 
and $k\not \equiv 6,18 \pmod*{25}$. Then
\begin{equation} \label{mainres}
\cD_k(n) = \begin{cases} \min 
\cS_{k,n} & \text{if~} k \equiv 0, 4\pmod*{5},  \\
\min \{\cS_{k,n} \cup \{ m \ge 5n/3: m \in \cB_{5,k}\}\} & \text{if~}k \equiv 1\pmod*{5}, \\
\min \{\cS_{k,n} \cup \{m \ge 5n/3:m \in \cA_{5,k}\cup \cB_{5,k} \}\} & \text{if~}k \equiv 3\pmod*{5},
\end{cases}
\end{equation}
with
$$\cA_{5,k}:  =
\{m = a \cdot 5^b : a \in \cA_k \text{~and~} b \ge 1\},~~\cB_{5,k} :=
\{m = a \cdot 5^b : a \in \cB_k \text{~and~} b \ge 1\}.$$
If $k\equiv 1,3\pmod*{5}$, $k\not \equiv 6,18 \pmod*{25}$ and $k>1$,
then  $\cD_k(n)=\min \cS_{k,n}$
if $n\ge n_p(\frac{5}{3})$ for
some odd prime divisor $p$ of $k(k+1)$. 
If $k>1$ is odd and $p$ divides $k$, it suffices to require that $n\ge n_p^{o}(\frac53)$.
\end{theorem}
\noindent Observe that on taking $k=1$ and $k=2$
we recover Theorem \ref{1+2}. 
Theorem \ref{main2} shows that the case  
$k\equiv 1\pmod*{3}$ is considerably 
more subtle than $k\not\equiv 1\pmod*{3}$. 
However, if $k\equiv 1\pmod*{3}$ then
$\cA_k$ and $\cB_k$ take the 
simpler form
$$\cA_k=\{m~{\text{odd}}:\text{if~}p\mid m,~{\text{then}}~p\mid
k\},\,\,\cB_k=\{m~{\text{even}}:\text{if~}p\mid m,~{\text{then}}~p\mid
k(k+1)\}.$$
Further using \eqref{excludedcongruenceclasses}
we deduce that Theorem \ref{main2} gives
a complete characterization of the discriminator for a set of integers $k$
having density $1-\frac{1}{3}\cdot
\frac{7}{25}=\frac{68}{75}$.
\par If $k$ is a power of two, then $\cA_k\cup \cB_k$ only contains $1$ as an odd number. It is thus natural to wonder about the parity of $\cD_{2^e}(n)$ for $n>1$. In this direction Theorem \ref{main2} leads to the following corollary.
\begin{cor}
If $n>1$, $e\ge 0$ with $e\not\equiv 4\pmod*{10}$, then
$\cD_{4^e}(n)$ is even.
\end{cor}
Finally, we observe that the interval $[n,3n/2)$ in Theorem \ref{oldmaincorrected} can often be replaced by a larger one. Theorem 
\ref{thm2refined} gives the details.
\subsection{The exceptional set $\cF_k$}
By \eqref{Dkn=n} we know that $\cA_k\cup\cB_k\subseteq \cD_k$.
Inequality \eqref{dkinequal2} suggests to consider
the \emph{exceptional set}
$$\cF_k:=\cD_k\setminus (\cA_k\cup\cB_k),$$ that is,
\begin{equation}
\label{ABF}
\cD_{k}=
\cA_k\cup
\cB_k\cup \cF_k,
\end{equation}
with $ \cF_k $ disjoint from both $ \cA_k $ and $ \cB_k$. 
\begin{lemma} 
\label{fktrivial}
Let $k>1$ be an integer.\hfil\break
{\rm a)} The set $\cF_k$ is finite.\hfil\break
{\rm b)} There are infinitely many $k$ for which the set $\cF_k$  is non-empty.\hfil\break
{\rm c)} The cardinality of the
set $\cF_k$ can be larger than any given bound.
\end{lemma}
\begin{proof}
a) For $k>2$ this is a direct consequence of Theorem \ref{oldmaincorrected} and
the definition of $\cF_k$. For $k=2$ it follows from 
Theorem \ref{1+2}b.\hfil\break
b)+c). The idea is to take
$k\equiv 1\pmod*{N!}$ with $N\ge 5$ large enough.
Then $\cD_k(n)=\cD_1(n)$ for
$n=1,\ldots,N,$
and thus the values involving $5$ will appear (cf. 
part a) of Theorem  \ref{1+2}). As
$5\nmid k(k+1),$ these 
are not in $\cA_k\cup
\cB_k,$ and so they must be 
in $\cF_k$. Since it can
be shown that infinitely many values of
$\cD_1$ are divisible by $5$, the proof
is completed. 
\end{proof}
\begin{table}[h]
\begin{tabular}{|c|c|c|c|c|}
\hline
$k$ & mod $25$ & $\cF_k$ ($a=125$) & $k(k+1)$ &  $n_p(5/3)$\\
\hline\hline
$16$ &  $16$ & $\{2a,4a,8a,16a,32a,64a\}$ & $2^4 \cdot \textbf{17}$ & $78644$\\
\hline
$73$ & $23$  &$\{a, 2a, 4a, 8a, 16a\}$ & $2 \cdot \textbf{37} \cdot 73$ & $1229$\\
\hline
$136$ & $11$ &$\{2a, 4a, 8a, 16a, 32a\}$ & $2^3 \cdot \textbf{17} \cdot 137$ & $78644$\\
\hline
$148$ &  $23$ & $\{a, 2a, 4a, 8a\}$ & $2^2 \cdot \textbf{37} \cdot 149$ & $1229$\\
\hline
$271$ &   $21$ & $\{2a ,4a,8a,16a,32a,64a\}$ & $2^4 \cdot \textbf{17} \cdot 271$ 
& $78644$\\ 
\hline
$283$ & $8$  & $\{a,2a,4a,8a,16a,32a\}$ & $2^2\cdot \textbf{71}\cdot 283$ & $4916$\\
\hline
$313$ & $13$ & $\{a,2a\}$ & $2\cdot \textbf{157}\cdot 313$ & $154$\\
\hline
    \end{tabular}
    \medskip \medskip
\caption{Some non-empty exceptional sets $\cF_k$}
\label{tab:125}
\end{table}
Tab.\,\ref{tab:125} demonstrates Lemma \ref{fktrivial}a. 
Every number 
appearing in it is of the form $2^b\cdot 5^3$ and 
explained by
Theorem \ref{main2} (which covers all the congruence 
classes 
mod $25$ appearing in the table).
The final column gives $n_p(\frac{5}{3})$ for
a prime $p$ indicated in bold in the column headed
$k(k+1)$. The number $\frac{5}{3}n_p(\frac{5}{3})$ is an upper bound for the largest number
in $\cF_k$. It is of crucial importance here to choose the right $p$, if for $k=136$ for example we would choose $p=137$, then we 
end up with
$n_{137}(5/3)=2516583$, whereas
for $p=17$ we obtain $n_{17}(5/3)=78644$.
The set given in the first row of Tab.\,\ref{tab:125} is 
certainly a subset of
$\cF_{16}$ by Tab.\,\ref{tab:16}.

 In 
Sec.\,\ref{effective} we
establish an effective, but 
unfortunately huge, upper bound for $\max \cF_k$. 
\begin{thm}
\label{largefkbound}
For $k>3$ we have
$$\max\cF_k\le 2^{k^{10^{10}\log\log k}}.
$$
\end{thm}

\subsection{Outline of the proof of Theorem \ref{main2}}
\label{proofstrategy}
For $m$ to be a potential discriminator value its \emph{rank of appearance} $z(m)$ (see 
Def.\,\ref{def:appear}) has to be large. The idea is now to first identify those values of $m$. This 
is the object of Secs.\,\ref{sec:powerappear}--\ref{sec:zlarge}, with basic
properties of $z(m)$ being recalled in Sec.\,\ref{sec:appearol}.
\par If $m$ is in $\cF_k,$ then there is a unique
prime power $p^e$ with $p\nmid k(k+1)$ such that $p^e$ 
exactly divides $m.$ 
We call $p^e$ a \emph{wild prime power}\footnote{This terminology is inspired 
by the novel ``The Wild Numbers: a Novel" by Philibert Schogt.} for $k$ (the smallest one being
$125$, cf.\,Tab.\,\ref{tab:125}). The major part of the proof of Theorem \ref{main2} consists of showing that $p=5$. 
This is the content of Theorem \ref{pequals5}.
The proof idea is to replace a wild prime power by a suitable number of the form
$2^a\cdot 5^b$ and thus get a smaller, but still discriminating, number. For this we need to ensure the existence of numbers of
the form $2^a\cdot 5^b$ in small enough intervals, a problem in the realm of Diophantine approximation. This is studied in 
Sec. \ref{sec:interval} (and in greater generality in 
Languasco et al.\,\cite{LLMT}).
\par Once we know that $p=5$ we are left with a very restricted set of potential discriminator values. The even
ones have good discriminating properties, but in general not the odd ones. To weed these further
out we use a more refined quantity, \emph{the incongruence index}, which unfortunately is
more awkward to work with than $z(m)$. It is studied
in Secs.\,\ref{sec:incong}--\ref{sec:discriminatory}, culminating in Lemma \ref{laatsteloodje}.
The proof of Theorem \ref{main2} now follows (in essence) on combining this lemma with Theorem \ref{pequals5}.
\begin{rem}
The congruence classes modulo $25$ covered by Theorem \ref{main2} are precisely the congruence classes for 
which $z(25)=15$ (see Tab.\,\ref{zkp2}) or $z(25)=25$.
\end{rem}

\subsection{Outline of the proof of Theorem 
\ref{largefkbound}}
Let $m\in {\mathcal F}_k$. Write $m=p_1^{e_1}\cdot m_1$, where $p_1\nmid k(k+1)$, $p_1^{e_1}$ exactly divides $m$, and $m_1$ consists only of prime factors of $k(k+1)$. We need to bound $p_1^{e_1}$ and $m_1$. We explain only the case 
when $m_1$ is odd as the even case is similar. Let $p$ be any odd prime factor of $k(k+1)$. It follows from uniform distribution theory that there exists $u$ such that 
 $$
 \left\{(u+1)\frac{\log 2}{\log p}\right\}\in \left(\frac{\log(4/3)}{\log p}, \frac{\log(3/2)}{\log p}\right). 
 $$
 This is containment \eqref{eq:u}. We show by an elementary argument that $p_1^{e_1}<2^{u+1}$. Thus, it suffices to bound $u$. This we do using the Koksma-Erd\H os-Tur\'an inequality which bounds the discrepancy of a sequence modulo $1$ by an exponential sum involving the distances to nearest integers of the members of our sequence. In our case, the members of our sequence are the multiples of $\log 2/\log p$, so we can bound the distances to the nearest integer using a version of Baker's lower bounds for linear forms in logarithms due to 
 Matveev (Theorem \ref{Matveev11}). Putting everything together gives a bound on $u$ in terms of $p$ which is exponential in $(\log p)(\log\log p)$ (see Lemma \ref{lem:10}). The argument can be iterated. Namely one writes 
 $m_1=\prod_{i=1}^s q_i^{a_i}$, with $q_1,\ldots,q_s$ divisors of $k(k+1)$,
 and one uses a similar argument to bound the exponents $a_1,\ldots,a_s$. A similar extra step is needed to bound the exponent of $2$ in case $m_1$ is even.

\subsection{Related work on other discriminators}
Apart from the infinite family of recurrence discriminators dealt
with here, only one other infinite family has
been studied, namely by Ciolan and Moree \cite{CM}.
Also in this case the associated discriminators $\cD$ have the property that
$\cD(2^b)=2^b$ for every $b\ge 1.$ De Clercq and his (many!)\,coauthors \cite{manyinterns} have classified all
binary linear recurrences for which the discriminator has
this property. It is expected that for all of them a rather
simple characterization of the discriminator should be possible. 
However, very little is known for
second order linear recurrences not of this form.

The work of de Clercq et al.\,\cite{manyinterns} was partly generalized by Ferrari
\cite{MatteoFerrari}, who, given any fixed odd prime $p$, found
a large class of binary recurrences 
$\bf a$ for which $\cD_{\bf a}(p^k)=p^k$ for every $k\ge 1.$

\section{Preliminaries}
\subsection{Notation}\label{subsec:notation}
The characteristic equation of the Shallit recurrence ${\bf U}(k)$ is
$$
x^2-(4k+2)x+1=0.
$$
Its roots are $\alpha(k)$ and $\alpha(k)^{-1},$ where
$$
\alpha(k)=2k+1+2{\sqrt{k(k+1)}},\quad \alpha^{-1}(k) = 2k + 1 - 2\sqrt{k(k+1)}.
$$
Note that
$$\alpha(k)= \beta^2(k),\quad 
\text{with}\quad\beta(k)= \sqrt{k+1} + \sqrt{k}.$$
The discriminant of the Shallit sequence is
\begin{equation*}
\Delta(k) := (\alpha(k)-\alpha(k)^{-1})^2=16k(k+1), 
\end{equation*}
and we easily verify that
\begin{equation}
\label{Binetform}
U_n(k) = \frac{\alpha^n(k) - \alpha^{-n}(k)}{\alpha(k) - \alpha^{-1}(k)} =
\frac{\beta^{2n}(k) - \beta^{-2n}(k)}{\beta^2(k) - \beta^{-2}(k)}.
\end{equation}
Given a prime $p$, we define
\begin{equation}
\label{epk}
e_p(k) := \left(\frac{k(k+1)}{p}\right),
\end{equation}
where $(\frac{\cdot}{p})$ is the Legendre symbol.

\subsection{The index of appearance}
\label{sec:appearol}
A crucial role in our considerations is played
by the \textit{index of appearance}.
\begin{definition}[Index of appearance]
\label{def:appear}
Let $k\ge 1$ be fixed. Given $m$, the smallest $n\ge 1$ such
that $m$ divides $U_n(k)$
exists and is called
the \textit{index of appearance} of $m$ in ${\bf U}(k)$ and
is denoted by $z_k(m)$.
\end{definition}
For notational convenience we suppress
the dependence of $z_k(m)$ on $k$ and, when there is no danger of
confusion, we denote it simply by $ z(m). $ 
The following result is trivial, but we will use it time and again.
\begin{lemma}
\label{zfactor2}
If $m=\cD_k(n),$ then $z(m)\ge n$ and $z(m)>m/2.$
\end{lemma}
\begin{proof} 
Since $U_{z(m)}\equiv U_0\pmod*{m}$ it follows that $z(m)\ge n$.
The second assertion we prove by contradiction and so
suppose that  $z(m)\le m/2.$ 
The interval $[z(m),2z(m))$ contains a power
of two, say $z(m)\le 2^b< 2z(m)\le m$. 
Since $2^b\ge z(m)\ge n$, it follows from
Corollary \ref{cor:twopower} that $U_0(k),\ldots,U_{n-1}(k)$ are  pairwise distinct modulo $2^b.$ As $2^b<m$, this contradicts the definition of the discriminator.
\end{proof}
Thus a way to characterize discriminator values would be
to first characterize those integers $m$ for which $z(m)>m/2$ (note that in part I we
already determined the integers $m$ for which $z(m)=m$, cf.\,Lemma 
\ref{appearance}).
This we address in Sec.\,\ref{sec:halfm}. The next step is then to investigate the discriminatory properties of
these $m$ (see Sec.\,\ref{sec:discriminatory}).
\par If $p$ divides $k(k+1)$ we have $z(p)=p$. This follows from the following
trivial lemma.
\begin{lem}
If $p\mid k$, then $U_n(k)\equiv n\bmod{p}$. If $p\mid (k+1)$, then 
$U_n(k)\equiv (-1)^{n+2}n\bmod{p}$.
\end{lem}
\begin{cor}
If $p\mid (k+1)$ is odd, then $U_{\frac{p-1}2}(k)\equiv U_{\frac{p+1}2}(k)\pmod*{p}$ and for $0\le i<j\le (p-1)/2$ we have 
$U_{i}(k)\not\equiv U_{j}(k)\pmod*{p}$.
\end{cor}
\subsection{The index of appearance in prime 
powers} 
\label{sec:powerappear}
The index of appearance in a prime power $p^b$ is
related to the multiplicative order of 
$\alpha$ modulo $p^b$.
\begin{lemma}
\label{zpb}
Let $p$ be odd such that $e_p(k)=-1$ and
let $b\ge 1$ be an integer. Then $z(p^b)$ is the minimal $m\ge 1$ such that $\alpha^m\equiv \pm 1\pmod*{p^b}$.
\end{lemma}
\begin{proof}
The proof of \cite[Lemma 5]{FLM} applies here verbatim, but
with $32$ replaced by $\Delta(k),$ and $\mathbb Z[\sqrt{2}]$ by
$\mathbb Z[\sqrt{k(k+1)}]$.
\end{proof}

The following lemma is basic and will be taken for granted in all 
our arguments involving the index of appearance.

\begin{lemma}[{\cite[Lemma 2]{FLM}}]
\label{lemma2FLM}
The index of appearance $z$ of the sequence $\mathbf{U}(k)$ has the following properties. 
\begin{enumerate}[{\rm(1)}]
\item If $p \mid U_m(k)$, then $z(p) | m;$ 
\item  If $p\mid k(k+1)$, then $z(p) = p;$ 
\item  If $p \nmid k(k+1)$, then $z(p) | p - e_p(k);$
\item $z(p^b) = p^{\max\{b-\nu_p(U_{z(p)}(k)), 0\}}z(p)$. 
In particular, $z(p^b) \mid p^{b-1}z(p);$
\item If $n=m_1  \cdots m_s$ with $m_1, \ldots, m_s$ pairwise coprime, then 
\begin{equation*}
    z(m_1 \cdots m_s) = \lcm[z(m_1), \ldots, z(m_s)].
\end{equation*}
\end{enumerate}
\end{lemma}

In part 4 we mostly have $z(p^b)=p^{b-1}z(p)$. In order to 
determine whether there can be exceptions to this, we introduce the notion of 
\emph{special prime}.
\begin{defi}[special prime]
A prime $p$ is said to be special if
$p\mid k(k+1)$ and $p^2\mid U_p(k)$.
\end{defi}
If $p$ is special, then $z(p^b)\mid p^{b-2}z(p)$ for every $b\ge 2$, otherwise
$z(p^b)=p^{b-1}z(p)$.
In part I (Lemma 3) it is shown that if $p$ is special, then we must have $p=3$.
It is easy to check that 3 is special if and only
if $k\equiv 2\pmod*{9}$ or $k\equiv 6\pmod*{9}.$
If 3 is special and 9 divides $m,$ then $z(m)<m.$
This together with Theorem \ref{oldmaincorrected} 
shows that 
\begin{equation}
\label{setequality}
\cA_k=\{m~{\text{odd}}:z(m)=m,
~m\in \cP(k)\}
\quad\text{and}\quad
\cB_k=\{m~\text{even}:z(m)=m\},
\end{equation}
where
$$\cP(k):=\{m\ge 1:p\mid m\Rightarrow p\mid k\}$$
is the set of
positive integers $m$ composed only of prime
factors dividing $k$.
\par The next lemma is formulated 
and proved in \cite[Sec.\,6.2.2]{FLM}, but not stated as 
a lemma there.
\begin{lemma} \label{Lemma2.8}
Let $p \nmid k(k+1)$ and $b\ge 1.$ Then
\begin{equation}
\label{divi}
z(p^b) \mid p^{b-1}(p - e_p(k))/2;
\end{equation}
moreover if $p \equiv 3\pmod*{4}$ and we assume
\begin{equation*}
\biggl( \frac{k+1}{p} \biggr) = 1 \quad \text{and} \quad \biggl( \frac{k}{p} \biggr) = -1,
\end{equation*}
then
\begin{equation*}
z(p^b) \mid p^{b-1}(p + 1)/4.
\end{equation*}
\end{lemma}
The next proposition shows that \eqref{divi} is sometimes sharp.
The various congruence classes
of $k$ counted in parts b,c and 
d are
explicitly worked out in 
Tab.\,\ref{zkp2} for the primes $3\le p\le 17$ (with 
a few exceptions where the table margins would 
be too small).
\begin{prop} \label{Florianapril2023} Let $p\ge 5$ be a prime not
dividing $k(k+1)$.
Put $f=\varphi(\frac{p+1}{2})$, with $\varphi$ Euler's totient function.
\begin{itemize}
\item[(a)] For $k$ there are exactly $f$ classes modulo $p$ such that
$z_k(p) =(p+1)/2$;
\item[(b)] For $k$ there are exactly $f$ congruence classes modulo $p^2$ such that $z_k(p^2)=(p+1)/2$;
\item[(c)] For $k$ there are exactly $(p-1)f$ congruence classes modulo $p^2$ such that $z_k(p^2)=p(p+1)/2$;
\item[(d)] For $k$ there are exactly $p-f-2$ congruence classes modulo $p$ such that $z_k(p)<(p+1)/2$.
\end{itemize}
\end{prop}
\begin{proof}
\par\noindent (a) We observe that if $k(k+1)$ is not a square modulo $p$, then $\alpha(k) = 2k+1 + 2\sqrt{k(k+1)}$ is quadratic modulo $p$. Here, by $\sqrt{}$ we mean any fixed determination of the square root. Thus, $\alpha(k)\in {\mathbb F}_{p^2}\backslash {\mathbb F}_p$, 
where 
${\mathbb F}_{p^2}$ is the unique  quadratic field over $p$ with $p^2$ elements. 
The Frobenius automorphism sends $\alpha(k)$ into its conjugate $2k+1-2{\sqrt{k(k+1)}}=\alpha(k)^{-1}$. Hence, $\alpha(k)^p=\alpha(k)^{-1}$ 
in ${\mathbb F}_{p^2}$, and so $\alpha(k)^{p+1}=1$ in ${\mathbb F}_{p^2}$. In particular, $\alpha(k)^{(p+1)/2}=\pm 1$ in ${\mathbb F}_{p^2}$. 
Further, by Lemma \ref{zpb}, $(p+1)/2$ must be the minimal $m$ such that $\alpha(k)^m=\pm 1$ in ${\mathbb F}_{p^2}$. Let $\rho$ be a primitive root modulo $p$. Write $\alpha(k)=\rho^d$ for some integer $d$. 
Then 
$\alpha(k)^{(p+1)/2}=\pm 1$ implies $\rho^{d(p+1)/2}=\pm 1$. Since $\rho$ is a primitive root, it follows that $p-1\mid d$. Thus, $d=(p-1)w$. Since $(p+1)/2$ is minimal such that $\alpha(k)^{(p+1)/2}=\pm 1$, 
it follows that $w$ is coprime to $(p+1)/2$. But $w\in [0,p+1]$. 
Each of the intervals $[0,(p+1)/2-1]$ and $[(p+1)/2, p+1]$ contains exactly $\phi((p+1)/2)$ numbers of the form $w$ which are coprime 
to $(p+1)/2$. For each one of these, $\alpha=\rho^{(p-1)w}$ is an element of ${\mathbb F}_{p^2}$. Then, keeping in mind
that $\alpha^{p^2-1}=1$, we see that $\alpha^p=\alpha^{p(p-1)w}=\alpha^{(p-1)(p+1-w)}$ and $p+1-w$ is also coprime to $p+1$. Thus, the $2\phi((p+1)/2)$ numbers get grouped 

into $\phi((p+1)/2)$ non-overlapping unordered pairs $\{\alpha,\alpha^p\}$. Let $t=\alpha+\alpha^p$. Then $t\in {\mathbb F}_p$ and  $(\alpha,\alpha^p)$ are roots of 
$$x^2-tx+1=0.
$$
It remains to see that we can choose $k$ such that  
$4k+2=t\pmod*{p}$, 
which is clear since $2$ is invertible modulo $p$.  
This gives the statement. 

\par\noindent (b) If $k$ is such that $z_k(p^2)=(p+1)/2$, then certainly $z_k(p)=(p+1)/2$. Thus, $k\pmod*{p}$ is one of the classes counted at part (a). It remains to prove that each such class can be lifted uniquely to a class modulo $p^2$ such that $z_k(p^2)=(p+1)/2$. But with a fixed $k$, putting $x:=2k+1$, we have
\begin{eqnarray*}
U_1(k) & = & 1;\\
U_2(k) & = & 2x;\\
U_3(k) & = & 4x^2-1;\\
U_4(k) & = & 8x^3-4x;\\
U_{n+2}(k) & = & 2x U_{n+1}(k)-U_n(k)\qquad {\text{\rm for ~all}}\qquad n\ge 3.
\end{eqnarray*}
We recognize that $U_{n}(x)$ is the Chebyshev polynomial $\sin(n\theta)/\sin(\theta)$ as a polynomial in $\cos(\theta)$,
which has discriminant $2^{(n-1)^2}n^{n-2}$, 
cf.~Dilcher and
Stolarsky \cite{DS}.
So, for us we have that $x=2k+1$ is a solution of $U_{(p+1)/2}(x)\equiv 0\pmod*{p}$, and we would like to extend it to a unique solution of the above congruence modulo ${p^2}$. This is possible via Hensel's lemma provided that $p$ does not divide the discriminant of $U_{(p+1)/2}(x)$ as a polynomial, which is the   case since this discriminant is 
$2^{((p-1)/2)^2} ((p+1)/2)^{(p-5)/2}$. This proves (b).

(c) is also immediate. By part (a), there are $f$ classes $k$ modulo $p$ for which $z_k(p)=(p+1)/2$. These $f$ classes give $pf$ lifts to classes modulo $p^2$. Exactly $f$ of them have the property that $z_k(p^2)=(p+1)/2$. Thus, for the 
remaining  
$(p-1)f$ classes, it must be the case  that $z_k(p^2)=p(p+1)/2$. 

(d) is also immediate. There are $p-2$ classes for $k$ modulo $p$ as we need to exclude $k\equiv 0,-1\pmod*{p}$ for which $k(k+1)$ is a multiple of $p$. By (a), there are $f$ of them for which $z_k(p)=(p+1)/2$. So there are $p-f-2$ of them for which 
$z_k(p)<(p+1)/2$. 
\end{proof} 
\begin{cor}
\label{freechoice}
For every odd prime $p$ there exists at least one congruence class modulo $p^2$ for $k$ such that $z_k(p^2)=(p+1)/2$ and one such that $z_k(p^2)=p(p+1)/2$.
\end{cor}

\begin{table}[h]
\begin{tabular}{|c|c|c|c|}
\hline
$p$ & $k$ & \text{congruence classes}  & \text{mod}\\
\hline\hline
$3$ & b & 4 & $9$\\
\hline
 & c &  $1,7$ & $9$\\
\hline
 & d &  none & $3$\\
\hline\hline 
$5$ & b & $6,18$ & $25$\\
\hline
 & c &  $1,3,8,11,13,16,21,23$ & $25$\\
\hline
 & d &  $2$ & $5$\\
\hline\hline 
$7$ & b & $2,46$ & $49$\\
\hline
 & c & $4,9,10,16,18,23,25,30,32,37,39,44$ & $49$\\
\hline
 & d & $1,3,5$ & $7$\\
\hline\hline
$11$ & b & $23,97$ & $121$\\
\hline
 & c &  $c_1,\ldots,c_{20}$ & $121$\\
\hline
 & d &  $2,3,4,5,6,7,8$ & $11$\\
\hline\hline
$13$ & b & $1,8,49,119,160,167$ & $169$\\
\hline
 & c &  $c_1,\ldots,c_{72}$ & $169$\\
\hline
 & d &  $3,5,6,7,9$ & $13$\\
 \hline\hline
$17$ & b & $20,53,111,177,235,268$ & $289$\\
\hline
 & c &  $c_1,\ldots,c_{196}$ & $289$\\
\hline
 & d &  $1,4,5,6,8,10,11,12,15$ & $17$\\
\hline 
    \end{tabular}
    \medskip \medskip
\caption{Congruence classes related to $z_k(p^2)$ 
(cf.\,Proposition \ref{Florianapril2023})}
\label{zkp2}
\end{table}

\begin{rem}
For $p\ge 5$ we have $p-\varphi((p+1)/2)-2\ge p-(p-1)/2-2\ge (p-3)/2\ge 1$ and so 
by Proposition \ref{Florianapril2023} there is
at least one congruence class modulo $p$ for $k$ such that $z_k(p)<(p+1)/2$.
\end{rem}
\begin{rem}
Proposition \ref{Florianapril2023} suggests 
considering Artin primitive root type problems such
as whether given $k$ the
set of primes $p$ such
that $z_k(p)=(p+1)/2$ has 
a natural density. Likely these questions can be answered
assuming the Generalized Riemann Hypothesis. These issues also play a role in understanding the behavior of $\rho_k,\sigma_k$ and $\tau_k$ (see Sec.\,\ref{subsub:closeup}).
We might come back
to this in a sequel to this paper.
\end{rem}
\subsection{Integers for which the index of appearance $z(m)$ satisfies $z(m)>m/2$}
In this section we characterize the integers
$m$ for which $z(m)>m/2$.
\label{sec:halfm}
\begin{lemma}
\label{lem:withp}
If $m/2<z(m)<m,$ then there exists a prime $p\nmid k(k+1),$ such that
\begin{equation}
\label{withp}
z(m)=\frac{m(p+1)}{2p}.
\end{equation}
Further, $z(p)=(p+1)/2$ and $e_p(k)=-1.$ 
The integer $m$ can be written as $m = a \cdot p^b$ with 
$a \in \cP(k(k+1)),$ $z(a)=a,$ $(a,p(p+1)/2)=1$ and $b\ge 1.$  If $b\ge 2,$ then
$z(p^2)=p(p+1)/2.$
\end{lemma}
\begin{proof}
Write $m = a \cdot p_1^{b_1}\cdots p_r^{b_r}$ with $a \in \cP(k(k+1))$ and $p_1, \ldots, p_r \nmid k(k+1)$ distinct 
primes. 
Note that the $p_i$ are odd primes. 
We have either $z(a)=a$ or $z(a)=a/3.$ In the latter
case $z(m)\le m/3$, and so $z(a)=a$ and hence $r\ge 1.$
Assume first that $r\ge 2$. 
By Lemma \ref{lemma2FLM} we have
$
z(m)\le z(a)z(p_1^{b_1})\cdots z(p_r^{b_r}),
$
and so on invoking Lemma  \ref{Lemma2.8}
we obtain the inequality
\begin{equation}
\label{zmm2}
\frac{z(m)}{m}\le  \left(\frac{p_1+1}{2p_1}\right)\cdots \left(\frac{p_r+1}{2p_r}\right)\le  \frac{2}{3} \cdot \frac{3}{5} <\frac{1}{2}.
\end{equation}
It follows that $m = a \cdot p^b$ with $a \in \cP(k(k+1))$, 
$p\nmid k(k+1)$ a prime, and $b \ge 1.$
This implies that $p\nmid a$.
If $e_p(k)=1,$ then $$z(m)=z(a\cdot p^b)\le a\cdot p^{b-1}(p-1)/2<m/2,$$
by 
Lemma \ref{Lemma2.8},
contradicting our assumption on
$z(m)$. If $e_p(k)=-1$ and $z(p)$ is a proper divisor of $(p+1)/2,$ then $$z(m)\le a \cdot p^{b-1}z(p)\le a\cdot p^{b-1} (p+1)/4\le m/2,$$ again 
contradicting
our assumption on $z(m),$
and hence  $e_p(k)=-1$ and $z(p) = (p+1)/2$.
We have $(a,(p+1)/2)=1$, since otherwise 
$$z(m)=\lcm(z(a),z(p^b))=\lcm(a,z(p^b))
\le p^{b-1}\lcm(a,(p+1)/2)\le a\cdot p^{b-1}(p+1)/4\le m/2.$$
Finally, we either have $z(p^2)=p(p+1)/2$ or 
$z(p^2)=(p+1)/2.$ For $b\ge 2,$ the latter case cannot occur
as then $z(m)\le m/p<m/2.$
\end{proof}
\begin{cor}
\label{cor:wppmotivation}
If $m$ is a discriminator value, then either
$m\in \cP(k(k+1))$, or $m=a\cdot p^b$ with $a\in \cP(k(k+1))$ and $p\nmid k(k+1)$ a prime satisfying $e_k(p)=-1$. Furthermore, 
if $m$ is even, then $p\equiv 1\pmod*{4}.$  
\end{cor}
\begin{proof}
This is an immediate consequence on recalling that if $m$ 
is a discriminator value, then $z(m)>m/2$ by Lemma \ref{zfactor2}.
\end{proof}
\begin{lemma} 
\label{lem:abc}
Let $p\nmid k(k+1)$ be a prime.\newline
{\rm a)} If $z(p^2)=(p+1)/2$, then
\begin{equation*}
z(m) = \frac{(p+1)m}{2p} \quad \iff \quad m = p \cdot a, \, \,  \, (a,p(p+1)/2) = 1,  \, \, \,  z(a)=a.
\end{equation*}
{\rm b)} If $z(p^2)=p(p+1)/2$, then
\begin{equation*}
z(m) = \frac{(p+1)m}{2p} \quad \iff \quad m = p^b \cdot a, \, \,  \,  \, \, \, \,  \, b \ge 1, \, \, (a,p(p+1)/2) = 1, \, \, z(a)=a.
\end{equation*}
{\rm c)} If $z(p)\ne (p+1)/2,$ then it never happens that $z(m) = m(p+1)/(2p)$.
\end{lemma}
\begin{proof}
Part c is a corollary of Lemma \ref{lem:withp} and
so are the $\Rightarrow$ directions of parts a and b. 
Now let us prove the $\Leftarrow$ direction for part b (the 
proof for part a being very similar and easier). By assumption $a$ and $p^b$
are coprime and so $z(m)=\lcm(z(a),z(p^b)).$ The assumption on $z(p^2)$ ensures that
$z(p^b)=p^{b-1}(p+1)/2.$ Since by 
assumption $z(a)=a$ and $(a,(p+1)/2)=1,$ we conclude that $z(m)=\lcm(a,p^{b-1}(p+1)/2)=m(p+1)/(2p).$ 
\end{proof}
Note that the value of $z_k(p^b)$ only depends on the congruence class of $k$ modulo $p^b.$
For a given odd prime $p$ it is thus a finite computation to
determine the corresponding 
congruence classes in each of the three cases (with the number 
of congruence classes already given in Proposition \ref{Florianapril2023}).
The results are recorded for the first few primes in 
Tab.\,\ref{zkp2}. Using this table, Lemma 
\ref{appearance} and Lemma \ref{lem:abc}, we
can then write down results similar to the lemma below (for
$p=5$).
\begin{lemma} Suppose that $5\nmid k(k+1).$\hfil\break
{\rm a)} If $k \equiv 6, 18\pmod*{25}$, then
\begin{equation*}
z(m) = \frac{3m}{5} \quad \iff \quad m = 5\cdot a, \, \,  \,  
(15,a)=1, \, \, \, 
a \in \cP(k(k+1)).
\end{equation*}
{\rm b)} If $k \equiv 1, 3, 8, 11, 13, 16, 21, 23\pmod*{25}$, then
\begin{equation*}
z(m) = \frac{3m}{5} \quad \iff \quad m = 5^b \cdot a, \, \,  \,   b \ge 1, \, \, \,  (15,a)=1, \, \, \,a \in \cP(k(k+1)). 
\end{equation*}
{\rm c)} If $k \equiv 2\pmod*{5}$, then it never happens that $z(m) = 3m/5$.
\end{lemma}
Here, and in general if $p\equiv 5\pmod*{6},$ we require that $3\nmid a$ (as $a$ and $(p+1)/2$ have to 
be coprime), and for such $a$ we have $z(a)=a$ if and only if $a \in \cP(k(k+1))$ by Lemma \ref{appearance}, and so a distinction of cases depending on whether 3 is a special prime or not is
unnecessary.

\subsection{Integers for which the index of appearance is large}
\label{sec:zlarge}
In this section we consider
how large $\frac{z(m)}{m}$ can be for 
$m$ with $z(m)<m$, something
quite relevant for us. The smaller $\frac{z(m)}{m}$ is, the less likely it is that $m$ occurs as a discriminator value (if $\frac{z(m)}{m}\le \frac12$, then certainly $m$ does not occur as a discriminator value). 
The following quantities 
(considered in detail in Sec.\,\ref{subsub:closeup}) will play a main role. 
Some sample values are given in Tab.\,\ref{tab:rhosigmatau}.
\begin{defi}[$\rho_k,\sigma_k,\tau_k$]
\label{def:three}
The supremum
\begin{equation*}
\sup_{p\text{~prime}}\Big\{\frac{z_k(f(p))}{f(p)}:
z_k(f(p))<f(p)\Big\},
\end{equation*}
we denote by 
$\rho_k$ if $f(p)=p$,
by $\sigma_k$ if $f(p)=p^2$,
and by $\tau_k$ if $f(p)=2p^2$.
\end{defi}
If $z(f(p))<f(p)$, then
$z(f(p))/f(p)\le (p+1)/2p$, where the upper bound is decreasing as a function of $p$. This implies that in order to verify that, say, $\rho_k=(q+1)/2q$, it suffices to show that $z(q)=(q+1)/2$ and that either $z(p)=p$ or $z(p)<(p+1)/2$ for every prime $3\le p<q$.
\par As $z(2^n)=2^n$,
it is enough to take the supremum over the odd primes only.
Since $z(p)/p\ge z(p^2)/p^2\ge z(2p^2)/2p^2$ we have
$$\rho_k\ge \sigma_k\ge \tau_k.$$

Using these quantities Theorem 
\ref{oldmaincorrected} can
be improved. In part I it was
already noted that the interval $[n,3n/2)$ occurring there can be replaced by the potentially larger interval $[n,n/\rho_k)$. 
We will show that $\rho_k$ can be replaced by $\sigma_k$. Since $\sigma_k<\rho_k$ for infinitely many $k$ (see Sec.\,\ref{subsub:closeup}), this is an improvement.
\begin{theorem}
\label{thm2refined}
Let $k>2$ be fixed. We have
\begin{equation*}
\cD_{k}(n)\le \min\{m\ge n:
m\in \cA_{k}\cup \cB_{k}\},
\end{equation*}
with equality if the interval $[n,n/\sigma_k)$ contains an integer 
$m\in \cA_{k}\cup \cB_{k}$. We have
$\sigma_k=2/3$ if $k\equiv 1,7\bmod{9}$, 
and $\sigma_k\le 3/5$ otherwise.
\par If $\cD_k(n)$ is even, then
$$\cD_{k}(n)\le \min\{m\ge n:
m\in \cB_{k}\},$$
with equality if the interval $[n,n/\tau_k)$ contains an integer 
$m\in \cB_{k}$.  If 
$k\equiv 0,2,4\pmod*{5}$, then 
$\tau_k\le 7/13$, and $\tau_k\le 3/5$ for general $k$. 
\end{theorem}

\begin{table}[h]
\begin{tabular}{|c|c|c|c|}
\hline
$k$ & ${\rho}_k$ & $\sigma_k$ & $\tau_k$\\
\hline\hline
1 & 2/3 & 2/3 & 3/5 \\
\hline
2 & 4/7 & 7/13 & 7/13 \\
\hline
3 & 3/5 & 3/5 & 3/5\\
\hline
4 & 2/3 & 4/7 & 7/13 \\
\hline
6 & 3/5 & 12/23 & 19/37 \\
\hline
23 & 3/5 & 3/5 & 3/5 \\
\hline
24 & 7/13 & 7/13 & 7/13 \\
\hline
31 & 2/3 & 6/11 & 9/17 \\
\hline
93 & 3/5 & 4/7 & 7/13 \\
\hline
$3202$ & $2/3$ & $2/3$ &  
$15/29$  \\\hline
    \end{tabular}
    \medskip \medskip
\caption{Some sample values of $\rho_k,\sigma_k$ and $\tau_k$}
\label{tab:rhosigmatau}
\end{table}
If $z(p)=(p+1)/2$ for some prime $p$, then we can replace
$\sup$ by $\max$ in Definition \ref{def:three}. If $z(p)$ is never equal to $(p+1)/2$, but
infinitely often to $z(p)=(p-1)/2$, then $\rho_k=1/2.$ If
$z(p)$ is never equal to $(p+1)/2$ and at most finitely often to
$(p-1)/2$, then $\rho_k<1/2$. The
same remarks 
hold, mutatis mutandis, for $\sigma_k$ and $\tau_k$,
cf.\,the next three lemmas.
\begin{lemma}[\cite{FLM}]
\label{zqqmaxrho}
Let $k\ge 1$ be fixed.\hfil\break
{\rm a)} Suppose that $z(q_1)=(q_1+1)/2$ for some prime $q_1.$ Let $q$ be
the smallest prime such that $z(q)=(q+1)/2.$ Then 
\begin{equation*}
\rho_k=\max\left\{\frac{z(p)}{p}:z(p)<p\right\}=\frac{q+1}{2q}.
\end{equation*}
{\rm b)} If there is no prime $q_1$ such that $z(q_1)=(q_1+1)/2$, then $\rho_k\le 1/2$.\hfil\break
{\rm c)} We have 
$\rho_k=2/3$ if $k\equiv 1\pmod*{3}$ and
$\rho_k\le 3/5$ otherwise.
\end{lemma}
\begin{lemma}
\label{zqqmaxsigma}
Let $k\ge 1$ be fixed.\hfil\break
{\rm a)} Suppose that $z(q_1^2)=q_1(q_1+1)/2$ for some prime $q_1.$ Let $q$ be
the smallest prime such that $z(q^2)=q(q+1)/2.$ 
Then 
\begin{equation*}
\sigma_k=\max\left\{\frac{z(p^2)}{p^2}:z(p^2)<p^2\right\}=\frac{q+1}{2q}.
\end{equation*}
{\rm b)} If there is no prime $q_1$ such that $z(2q_1^2)=q_1(q_1+1)/2$, then $\tau_k\le 1/2$.\hfil\break
{\rm c)} 
We have
$\sigma_k=2/3$ if $k\equiv 1,7\bmod{9}$, 
and $\sigma_k\le 3/5$ otherwise.
\end{lemma}
\begin{proof}
We leave this to the reader, cf.\,the very similar (but more complicated) proof of Lemma \ref{zqqmaxtau}. Only part c needs special attention, here we use Proposition \ref{Florianapril2023} and  Tab.\,\ref{zkp2}.
\end{proof}
\begin{lemma}
\label{zqqmaxtau}
Let $k\ge 1$ be fixed.\hfil\break
{\rm a)} Suppose that $z(2q_1^2)=q_1(q_1+1)$ for some prime $q_1.$ Let $q$ be
the smallest prime such that $z(2q^2)=q(q+1).$ 
Then $q\equiv 1\pmod*{4}$ and 
\begin{equation*}
\tau_k=\max\left\{\frac{z(2p^2)}{2p^2}:z(2p^2)<2p^2\right\}=\frac{q+1}{2q}.
\end{equation*}
{\rm b)} If there is no prime $q_1$ such that $z(2q_1^2)=q_1(q_1+1)$, then $\tau_k\le 1/2$.\hfil\break
{\rm c)} We have $\tau_k\le 3/5$. If 
$k\equiv 0,2,4\pmod*{5}$, then 
$\tau_k\le 7/13$.
\end{lemma}
\begin{proof} 
Put $\rho(p)=z(2p^2)/(2p^2).$
If $p>q$ and $\rho(p)<1,$ then
$$\rho(p)\le\frac{p+1}{2p}<\frac{q+1}{2q}=\rho(q),$$ 
If $p<q$ and $\rho(p)<1,$ then $z(2p^2)\mid p(p-1),$  
$z(2p^2)\mid p(p+1)/2$ or $z(2p^2)\mid p+1.$
Since, respectively,
\begin{equation}
\label{threecases}    
\rho(p)\le\frac{p-1}{2p}<\frac12,\quad
\rho(p)\le\frac{p+1}{4p}<\frac12\quad\text{or}\quad\rho(p)\le\frac{p+1}{2p^2}<\frac12
\end{equation}
and $(q+1)/(2q)>1/2,$ we have established 
that $\tau_k=(q+1)/(2q).$ In case $q\equiv 3\pmod*{4},$
then $z(2)$ and $z(q^2)$ are both even and so $\rho(q)<1/2.$ Thus
$q\equiv 1\pmod*{4}.$\hfil\break
b) In this case, cf.\,\eqref{threecases}, we have $\rho(p)<\frac12$ for every prime $p$ and so the supremum is $\le \frac12$.\hfil\break
c) We apply parts a and b together with the observation that $\rho(50)\in \{\frac15,1\}$ if $k\equiv 0,4\bmod{5}$ (by 
Lemma \ref{lemma2FLM}) and
$\rho(50)\le \frac25$ if $k\equiv 2\bmod{5}$ (by Lemma \ref{Lemma2.8}).
\end{proof}
\par The following extends 
\cite[Lemma 14]{FLM} with some extra statements
involving the set $\cM$.
\begin{lemma}
\label{appearance}  Let $k\ge 1$.
We have $z(m)=m$ if and only if
$$
\begin{cases}
m\in \cP(k(k+1)),~9\nmid m;\cr
m\in \cP(k(k+1)),~9\mid m,~
\text{and~}3{~is~not~special.}
\end{cases}
$$
The remaining integers $m$ 
satisfy
$z(m)\le \rho_k m,$
with $\rho_k=2/3$ if $k\equiv 1\pmod*{3}$ and
$\rho_k\le 3/5$ otherwise.
Let $\mathcal M$ be the set of integers that 
are divisible by some prime square $p^2$ with $p\nmid k(k+1)$ a prime.
The integers $m$ in $\mathcal M$ 
satisfy
$z(m)\le \sigma_k m,$
with $\sigma_k=2/3$ if $k\equiv 1,7\bmod{9}$, 
and $\sigma_k\le 3/5$ otherwise. The
even integers $m$ in $\mathcal M$ satisfy $z(m)\le \tau_k m\le 3m/5$. If 
$k\equiv 0,2,4\pmod*{5}$, then 
$\tau_k\le 7/13$. 
\end{lemma}
\begin{cor}
\label{ktimeskplus1}
Suppose that $k\equiv 1\pmod*{3}.$ Then $z(m)=m$ if and only
if $m\in \cP(k(k+1)).$
\end{cor}
\begin{proof}[Proof of Lemma \ref{appearance}]
Only the statements involving $\cM$ need
to be proved. Let $b\in \cM$. We have
$$\frac{z(b)}b\le \sup_{m\in \cM}\Big\{\frac{z(m)}m\Big\}=\sup_{m\in \cM_1}\Big\{\frac{z(m)}m\Big\},$$
where $\cM_1$ is the set of integers divisible by at most one prime square $p^2$ with
$p\nmid k(k+1)$ (cf.\,the beginning of the proof of Lemma \ref{lem:withp}). 
Thus every
$m\in \cM_1$ is of the form 
$m=p^e\cdot m_1$, with $p\nmid k(k+1)$ some prime, 
$e\ge 2$ and $z(m_1)=m_1$. Since
$$\frac{z(m)}m\le \frac{z(p^2)}{p^2}\,\frac{p^{e-2}}{p^{e-2}}\,\frac{z(m_1)}{m_1}\le \frac{z(p^2)}{p^2},$$
we obtain that $z(b)/b\le \sigma_k$. If $m$ is even, then $m=p^e\cdot 2^f\cdot m_1$, with $p\nmid k(k+1)$, 
$e\ge 2$, $f\ge 1$ and $z(m_1)=m_1$, and
we have
$$\frac{z(m)}m\le \frac{z(2p^2)}{2p^2}\,\frac{p^{e-2}}{p^{e-2}}
\,\frac{2^{f-1}}{2^{f-1}}\,\frac{z(m_1)}{m_1}\le \frac{z(2p^2)}{2p^2},$$
and hence $z(b)/b\le \tau_k$. The proof
is concluded on invoking Lemmas 
\ref{zqqmaxrho}c, \ref{zqqmaxsigma}c and 
\ref{zqqmaxtau}c.
\end{proof}
\subsubsection{The numbers $\rho_k,\sigma_k$ and $\tau_k$: a close-up}
\label{subsub:closeup}
We investigate when there exists 
an integer $k$ such that
\begin{equation}
\label{image}
(\rho_k,\sigma_k,\tau_k)=\Big(\frac{p_1+1}{2p_1},\frac{p_2+1}{2p_2},\frac{p_3+1}{2p_3}\Big),\quad \text{~with~}p_1,p_2,p_3\text{~prescribed~primes}.   
\end{equation}
The primes $p_1,p_2$ and $p_3$ are not required to be distinct. Our main tool is Corollary \ref{freechoice}, which we will take for granted in the remainder of this section.
\begin{lemma}
\label{possibletriples}
The equation \eqref{image} has a solution $k$ if and only if $3\le p_1\le p_2\le p_3$
and 
$p_3\equiv 1\bmod{4}$. If $p_2\equiv 1\bmod{4}$, then we require in addition that  
$p_3=p_2$. If \eqref{image} is satisfied for some $k$, then it is satisfied for a positive density of integers $k$.
\end{lemma}
\begin{proof}
Since $z(2^n)=2^n$ the supremum is assumed in an odd prime and so $p_1,p_2$ and $p_3$ are odd. 
Since $\rho_k\ge \sigma_k\ge \tau_k$ we have 
$p_1\le p_2\le p_3$. 
\par The density assertion follows on noting that if $k_1\equiv k\bmod{p^2}$ for every odd prime $p\le p_3$, then 
$(\rho_{k_1},\sigma_{k_1},\tau_{k_1})=(\rho_k,\sigma_k,\tau_k)$ (observe that $z_k(p^2)$ only depends on the residue class of $k$ modulo $p^2$).
\par Note that \eqref{image} entails that $z(p_1)=(p_1+1)/2$, $z(p_2^2)=p_2(p_2+1)/2$
and $z(2p_3^2)=p_3(p_3+1)$. 
The latter identity forces $p_3$ to be 
congruent to $1\bmod{4}$ by Lemma \ref{zqqmaxtau}a.
If $p_2\equiv 1\bmod{4}$, then 
$z(2p_2^2)=p_2(p_2+1)$.
We have $z(2p^2)/2p^2\le z(p^2)/p^2<z(p_2^2)/p_2^2$ for 
$3\le p<p_2$. We conclude that
$\tau_k=(p_2+1)/2p_2$ and hence
$p_3=p_2$. 
\par It remains to prove 
that if the conditions on the primes $p_1,p_2$ and $p_3$ are satisfied, there exists a $k$ solving \eqref{image}. We take $k\equiv 0\bmod{p}$ for the odd primes $p<p_1$. If $p_2=p_1$ we choose $k$ to be in a residue class modulo $p_1^2$ such that $z(p_1^2)=p_1(p_1+1)/2$.
If $p_2>p_1$, we choose $k$ to be in a residue class modulo $p_1^2$ such that $z(p_1^2)=(p_1+1)/2$, 
$k\equiv 0\bmod{p}$ for the primes $p_1<p<p_2$ and $k$ in a residue class modulo $p_2^2$ such that 
$z(p_2^2)=p_2(p_2+1)/2$. If $p_2\equiv 1\bmod{4}$, 
then $p_3=p_2$ and we are done. Otherwise we take
$k\equiv 0\bmod{p}$ for the primes $p_2<p<p_3$ with
$p\equiv 1\bmod{4}$ and take $k$ in a residue class modulo $p_3^2$ for which $z(p_3^2)=p_3(p_3+1)/2$.
\end{proof}
\begin{exam}
By Lemma \ref{possibletriples} there exists a solution to 
\eqref{image} with $(p_1,p_2,p_3)=(3,7,13)$. 
We will now find such a solution. We take $k\equiv 0\bmod{3}$ and $k\equiv 18\bmod{25}$. This ensures that $\rho_k=\frac35$ and 
$\sigma_k\ge \frac47$. On requiring that $k\equiv 44\bmod{49}$, it follows that $\sigma_k=\frac47$. 
As $\tau_k\ne \frac47$ and $\tau_k\ne \frac6{11}$, we have $\tau_k\ge \frac7{13}$. We choose $k\equiv 2\bmod{13}$ and $k\not\equiv 119\bmod{169}$ to ensure that 
$\tau_k=\frac{7}{13}$. Finally, one checks that $k=93$ satisfies all the requirements, and we conclude that
$(\rho_{k},\sigma_{k},\tau_{k})=(\frac35,\frac47,\frac7{13}).$ 
By computer calcuation one can verify that $368, 431, 543$ and $606$ are the only other $k<1000$ having this property.
\end{exam}

\subsection{The incongruence index}
\label{sec:incong}
Apart from $z_k$ we will also make use of
the \emph{incongruence index} $\iota_k$, which was
introduced in Moree and 
Zumalac\'arregui \cite{PA}. It will allow us to rule out many odd values of $m$ with $z(m)=m$ as discriminator values (cf.\,Lemma \ref{laatsteloodje}).
\begin{defi}[incongruence index]
Given 
an integer $m\ge 1$, the incongruence index $\iota_k(m)$ is the largest integer $j$ such that
$U_0(k),\ldots,U_{j-1}(k)$ 
are pairwise distinct modulo $m.$ 
\end{defi}
Note that $\iota_k(m)\le z_k(m)$. 
In practice frequently $\iota_k(m)<z_k(m)$, which shows that the following, easy to prove, variant of Lemma 
\ref{zfactor2} is often stronger.
\begin{lemma}
\label{ifactor}
If $m=\cD_k(n),$ then $\iota_k(m)\ge n$ and $\iota_k(m)>m/2.$
\end{lemma}
The general idea is to use
$z(m)$ whenever possible and if it proves itself too weak a tool, then try to work with $\iota(m)$.
\begin{lem}
\label{iota5b}
Let $b\ge 1$ be an integer. Then
$$\iota_k(5^b)=
\begin{cases}
(3 \cdot 5^{b-1}+1)/2 & \text{~if~}k\equiv 1\pmod*{5}\text{~and~}k\not\equiv 6\pmod*{25};\cr
3\cdot 5^{b-1} & \text{~if~}k\equiv 3\pmod*{5}\text{~and~}k\not\equiv 18\pmod*{25}.
\end{cases}
$$
\end{lem}
\begin{proof}
If $k\equiv 1,3\pmod*{5}$, then 
$e_5(k)=-1$ and so $z(5)=3$. 
A trivial computation gives 
$U_3(k)=16k^2+16k+3$.
This is a multiple of $5$ for $k\equiv 1,3\pmod*{5}$, but not of $25$ since the classes $k\equiv 6,18\pmod*{25}$ are excluded. So, $5\mid U_3(k)$, but $25\nmid U_3(k)$
and by Lemma 10b it
follows that $z(5^{b})=3\cdot 5^{b-1}$, cf.\,Table \ref{zkp2}.
\par Now let us investigate when $U_i(k)\equiv U_j(k) \pmod*{5^b}$. 
Writing $\alpha, \alpha^{-1}$ for the roots of the 
characteristic equation
$x^2-(4k+2)x+1$,
by \eqref{Binetform} we need $\alpha^i-\alpha^{-i}\equiv \alpha^j-
\alpha^{-j} \pmod*{5^b}$, which
on multiplication by $\alpha^{i+j}$
yields 
$(\alpha^i-\alpha^j)(\alpha^{i+j}+1)\equiv 0\pmod*{5^{b}}$.
So, $5$ divides either $\alpha^i-\alpha^j$ or 
$\alpha^{i+j}+1$ ($5$ is inert in 
$\mathbb Z[\alpha]$ as $e_5(k)=-1$). When $k\equiv 3\pmod*{5}$ the second case doesn't happen. That is, for $k\equiv 3\pmod*{5}$ we have that $\alpha$ is one of $2\pm {\sqrt{3}}\pmod*{5}$ (the characteristic equation only depends on $k$ modulo $5$).  
Then $\alpha^2=2\pm {\sqrt{3}}$ and $\alpha^3\equiv 1\pmod*{5}$. So, we see that $-1\pmod*{5}$ is not in the multiplicative group generated by $\alpha\pmod*{5}$ 
when $\alpha=3\pmod*{5}$. Thus, $U_i(k)\equiv U_j(k)\pmod*{5^b}$
forces $\alpha^i\equiv \alpha^j \pmod*{5^b}$, so 
$\alpha^{i-j}\equiv 1\pmod*{5^b}$, so $U_{i-j}\equiv 0
\pmod*{5^b}$ (assuming say $i>j$), so $z(5^b)=3\cdot 5^{b-1}$ divides $i-j$. This takes care of $\iota_k(5^b)$ in 
case $k\equiv 3\pmod*{5}$.

In case $k\equiv 1\pmod*{5}$, there is no 
$i,j$ such that both $\alpha^{i}-\alpha^j\equiv 0 
\pmod*{5}$ and $\alpha^{i+j}+1\equiv 0\pmod*{5}$. To see 
why, assume there are such. Then $\alpha^{i+j}\equiv 
-1\pmod*{5}$ and $\alpha^{i-j}\equiv 1\pmod*{5}$. But when 
$k\equiv 1\pmod*{5}$, then $\alpha=3\pm 2{\sqrt{2}}$. Now 
$\alpha^2\equiv 2\pm 2{\sqrt{2}}\pmod*{5}$ and $\alpha^3\equiv -1\pmod*{5}$. So, the order of $\alpha$ modulo 5 is exactly 6 so asking of $i,j$ such that 
$\alpha^{i-j}\equiv 1\pmod*{5}$ and $\alpha^{i+j}\equiv -1\pmod*{5}$ gives
$i-j\equiv 0 \pmod*{6}$ and $i+j\equiv 3 \pmod*{6}$. Summing them we get $2i\equiv 3 \pmod*{6}$, which is false (there is no such 
$i$ with $2i\equiv 3\pmod*{6}$).

So, when $k\equiv 1\pmod*{5}$ and $k\not\equiv 6\pmod*{25}$, either $5^b$ divides $\alpha^{i-j}-1$ or 
$5^b$ divides $\alpha^{i+j}+1$. 
In the first case $i-j$ is a multiple of $z(5^b)$, so at least $3\cdot 5^{b-1}$. The second case gives 
$\alpha^{i+j}\equiv -1 \pmod*{5^b}$ so $i+j$ is an odd 
multiple of $3\cdot 5^{b-1}$. The extreme case is 
$i+j=3\cdot 5^{b-1}$ and we see that if $i>j$, 
then $i\ge 3\cdot 5^{b-1}/2$, so $i\ge (3\cdot 5^{b-1}+1)/2$.
\end{proof}
The next lemma studies to what extent a 
relatively small incongruence index remains relatively small after lifting to a larger modulus. Recall the definition \eqref{epk} of $e_p(k)$.
\begin{lem}
\label{incongruence1}
Suppose that $p\nmid k(k+1)$ and $\left(\frac{k+1}{p}\right)=-1$.
If $\iota_k(p^a)<p^a/2$ for some $a\ge 1$, then
$\iota_k(p^b)<p^b/2$ for all $b\ge a$. Furthermore, if $m$ is odd and $z_k(m)=m$, then 
$\iota_k(p^a\cdot m)<p^a\cdot m/2$.
\end{lem}
\begin{proof}
If $e_p(k)=1$,
then $z(p^{a})\le p^{a-1}(p-1)/2<p^a/2$ for all $a\ge 1$ by Lemma \ref{Lemma2.8}. So
\begin{equation}
\label{eq:1}
\left(\frac{k}{p}\right)=1,\quad \left(\frac{k+1}{p}\right)=-1.
\end{equation}
We show that $z_k(p^a)\le \lfloor p^{a-1}(p+1)/4\rfloor<p^a/2$ for all $a\ge 2$. Indeed, write $U_i\equiv U_j\pmod*{m}$ with $m$ coprime to $k(k+1)$ as 
$$
\alpha^i-\alpha^{-i}\equiv \alpha^j-\alpha^{-j}\pmod*{m},
$$
which is equivalent to 
$$
(\alpha^i-\alpha^j)(\alpha^{-(i+j)}+1)\equiv 0\pmod*{m}.
$$
It suffices that $\alpha^{i+j}\equiv -1\pmod*{m}$. Using \eqref{eq:1} we note
that
\begin{eqnarray*}
\alpha^{(p+1)/2} &=& \left(({\sqrt{k+1}}+{\sqrt{k}})^2\right)^{(p+1)/2}=({\sqrt{k+1}}+{\sqrt{k}})^{p+1}\\
& =& ({\sqrt{k+1}}+{\sqrt{k}})^p({\sqrt{k+1}}+{\sqrt{k}})\equiv ({\sqrt{k+1}}^p+{\sqrt{k}}^p)({\sqrt{k+1}}+{\sqrt{k}})\pmod*{p}\\
& \equiv & (-{\sqrt{k+1}}+{\sqrt{k}})({\sqrt{k+1}}+{\sqrt{k}})\equiv k-(k+1)\equiv -1\pmod*{p}.
\end{eqnarray*}
So, if we choose $m=p$ and $i+j=(p+1)/2$, then we have that $U_i\equiv U_j\pmod*{p}$. More generally, we can choose $m=p^a$. Since 
$$
\alpha^{(p+1)/2}\equiv -1\pmod*{p},
$$
we get that
$$
\alpha^{p^{a-1}(p+1)/2}\equiv -1\pmod*{p^a},
$$
and we can see that $i_k(p^a)\le p^{a-1}(p+1)/4$. Indeed, if $p^{a-1}(p+1)/2=2\ell+1$, then $U_{\ell}\equiv U_{\ell+1}\pmod*{p^a}$, so $i_k(p^a)\le \ell=\lfloor p^{a-1}(p+1)/4\rfloor$. If $p^{a-1}(p+1)/2=2\ell$, then we have 
$U_{\ell-1}\equiv U_{\ell+1}\pmod*{p^a}$, so $i_k(p^a)\le \ell=\lfloor p^{a-1}(p+1)/4\rfloor$. So, at any rate in this case $i_k(p^a)<p^a/2$ for all $a\ge 1$.
\par We now turn our attention to the final assertion. 
Let $i,j$ be such that $i<j<p^{\alpha-1}(p+1)/4+1$ and $U_i\equiv U_j\pmod*{p^{\alpha}}$. Then $U_{mi}\equiv U_{mj}\pmod*{p^{\alpha}m}$. To see this, note that both sides are $0$ modulo $m$ since $m$ divides
$U_{z(m)}=U_m$ and 
$U_m$ divides $\gcd(U_{mi},U_{mj})$. As for the divisibility by $p^{\alpha}$, the congruence 
$$
U_{mi}\equiv U_{mj}\pmod*{p^{\alpha}}
$$
is implied by 
$$
(\alpha^{mi}-\alpha^{mj})(\alpha^{-m(i+j)}+1)\equiv 0\pmod*{p^{\alpha}}
$$
which holds because 
$$\alpha^i-\alpha^j\mid \alpha^{mi}-\alpha^{mj}, \qquad \alpha^{-(i+j)}+1\mid \alpha^{-m(i+j)}+1
$$ 
($m$ is odd) and 
$$
(\alpha^i-\alpha^j)(\alpha^{-(i+j)}+1)\equiv 0\pmod*{p^{\alpha}}.
$$
Since certainly $im<jm\le \lfloor p^{\alpha-1}(p+1)/4\rfloor+1\rfloor m<p^{\alpha} m/2$, the proof is finished. 
\end{proof}
\begin{cor}
\label{incongruence3mod4}
Suppose that $p\nmid k(k+1)$ and $p\equiv 3\pmod*{4}.$
If $\iota_k(p^a)<p^a/2$ for some $a\ge 1$, then
$\iota_k(p^b)<p^b/2$ for all $b\ge a$. Furthermore, if $m$ is odd and $z_k(m)=m$, then 
$\iota_k(p^a\cdot m)<p^a\cdot m/2$.
\end{cor}
\begin{proof}
If $\left(\frac{k+1}{p}\right)=1$,
then $\left(\frac{k}{p}\right)=-1$ (as $e_p(k)=-1$). Then by Lemma \ref{Lemma2.8} we
have $\iota_k(p^a)\le z_k(p^a)\le p^{a-1} (p+1)/4<p^a/2$ for 
every $a\ge 1$.
\end{proof}
We expect that this corollary also holds for the primes
$p\equiv 1\pmod*{4}$, but is more difficult to prove.
As we do not need this 
generalization, we leave it to a possible
sequel to this paper.

\subsection{The discriminatory properties
of $m$ with large $z(m)$}
\label{sec:discriminatory}
The goal of this section is to prove Lemma 
\ref{laatsteloodje}. To this end we need the
next fundamental lemma and Lemma \ref{lem:last}
\begin{lem}[Lemmas 15 and 16 of \cite{FLM}]
\label{15+16}
Let $p$ be an odd prime and $b\ge 1$ be arbitrary.\hfil\break
If $p$ divides $k$, then $U_i(k)\equiv U_j(k)\pmod*{p^b}$ if and only if $i\equiv j\pmod*{z_k(p^b)}$.\hfil\break
If $p$ divides $k+1$, then $U_i(k)\equiv U_j(k)\pmod*{p^b}$ is equivalent to 
one of the following:
\begin{itemize}
\item If $i\equiv j\pmod*{2}$, then $i\equiv j\pmod*{z_k(p^b)}$; 
\item If $i\not\equiv j\pmod*{2}$, then $i\equiv -j\pmod*{z_k(p^b)}$.
\end{itemize}
\end{lem}

\begin{lemma}
\label{lem:last}
Assume that $m=2^ap^e$ with 
$a,e\ge 1$ is such that 
$e_p(k)=-1$, $p\equiv 1\pmod*{4}$ and $z(p)=(p+1)/2$. Then
$
U_i\equiv U_j\pmod*{m}
$
holds if and only if
$i\equiv j\pmod*{z(m)}$.
\end{lemma}
\begin{proof} 
A minor variation of the proof of \cite[Lemma 9]{FLM}. It
rests only on \cite[Lemma 4]{FLM} and \cite[Lemma 5]{FLM},
where now we should take Lemmas \ref{zpb} and \ref{Lemma2.8} 
from this paper. In addition a minor correction in the
argument has to be made, as sketched in Sec.\,\ref{proofcorrected}.
\end{proof}

\begin{lem}
\label{laatsteloodje}
Assume that 
$k\equiv 1\pmod*{3}$ and $z_k(25)=15$. Suppose that 
$m=5^e\cdot m_1$ with $e\ge 1$ and 
$z(m_1)=m_1$.
Then $\iota(m)=z(m)=3m/5$ if
\begin{itemize}
    \item $m_1\in \cB_k$; or
    \item $m_1\in \cA_k$ and $k\equiv 3\pmod*{5}$,
\end{itemize}
and otherwise $m$ is not a discriminator value assumed by $\cD_k$.
\end{lem}
\begin{proof}
Note that $z(m)=3m/5$. We first consider the case where $m_1$ is even. 
Note that this is equivalent with
$m_1\in \cB_k$, as
$\cB_k=\{m_1~\text{even}:z(m_1)=m_1\}$ by \eqref{setequality}.
Write $m_1=2^a\cdot m_2$ with $m_2$ odd.
Since $z(2^a\cdot 5^e)=2^a\cdot 3\cdot 5^{e-1}$, we have 
$U_i(k)\equiv U_j(k)\pmod*{2^a\cdot 5^e}$ if and only if 
 $i\equiv j\pmod*{2^a\cdot 3\cdot 5^{e-1}}$ by Lemma \ref{lem:last}.
As $m_2\in {\mathcal P}(k(k+1))$,
it follows by Lemma \ref{15+16} that
$U_i(k)\equiv U_j(k)\pmod*{m_1}$ 
if and only if $i\equiv j\pmod*{m_1}$.
Taken together these two 
equivalences show that 
$U_i(k)\equiv U_j(k)\pmod*{m}$ if and only if $i\equiv j\pmod*{3m/5}$.
Thus in case $m_1$ is even, we conclude that $\iota(m)=3m/5$. 
\par It remains to deal with the case where $m_1$ is odd. Since by assumption $z(25)=15$, we have $k\equiv 1,3\bmod{5}$ (see Table \ref{tab:125}).
\par First case: $k\equiv 1\pmod*{5}$. The assumption $z(25)=15$ ensures that 
$k\not\equiv 6\pmod*{25}$. By 
Lemma \ref{iota5b} and 
Lemma \ref{incongruence1} (with 
$p=5$) it then follows that $\iota(m)\le m/2$ and so $m$ is not a discriminator value.
\par Second case: $k\equiv 3\pmod*{5}$. Suppose that $m_1$ has an odd prime divisor that also divides $k+1$. Now write $m_1=p^a\cdot m_2$ with $p\nmid m_2$. Clearly $z(p^a)=p^a$. 
Set $i=(p^a-1)\cdot m_2\cdot 3\cdot5^{e-1}/2$ and $j=(p^a+1)\cdot m_2\cdot 3\cdot5^{e-1}/2$. Then $i\not\equiv j\pmod*{2}$ and 
$p^{a}\mid (i+j)$. Thus, $U_i(k)\equiv U_j(k)\pmod*{p^{a}}$
by Lemma \ref{15+16}. This lemma also implies 
that $U_i(k)\equiv U_j(k)\equiv U_0(k)\equiv 0\pmod*{m_2}$ as $i\equiv j\equiv 0\pmod*{m_2}$ and
$m_2\in {\mathcal P}(k(k+1))$.
The proof of Lemma 
\ref{iota5b} shows that 
if $i\equiv j\pmod*{3\cdot 5^{e-1}}$,
then $U_i(k)\equiv U_j(k)\pmod*{5^e}$.
We infer that $U_i(k)\equiv U_j(k)\pmod*{p^a\cdot m_2\cdot 5^e}$,
and hence
if $m$ discriminates the 
numbers $U_0(k),\ldots,U_{n-1}(k),$ then
$n\le (p^{a}+1)m_2\cdot 3\cdot 5^{e-1}/2.$
The interval $[(p^a+1)/2,p^{a})$ contains a power of $2$, say $2^b$. Then 
$2^b\cdot m_2\cdot 5^e$ is a better discriminator than $p^{a}\cdot m_2\cdot 5^e=m$. 
We conclude that if $m_1\not\in \cA_k$,
then $m$ is not a discriminator 
value. 
\par It remains to deal with the
case where $m_1\in \cA_k$.
The proof of Lemma 
\ref{iota5b} shows that 
$U_i(k)\equiv U_j(k)\pmod*{5^e}$
if and only
if $i\equiv j\pmod*{3\cdot 5^{e-1}}$. 
This in combination with Lemma 
\ref{15+16} shows
that $U_i(k)\equiv U_j(k)\pmod*{m}$ if and only if $i\equiv j\pmod*{3m/5}$. Hence $\iota(m)=3m/5$.
\end{proof}

\subsection{Intervals containing special integers}
\label{sec:interval}
We will first discuss how to compute $n_p(\alpha)$ (cf.\,Def.\,\ref{def:xpalpha}).
From basic Diophantine approximation we know
there exist $e,f,g$ and $h$ such that 
$$1<\frac{p^f}{2^e}<\alpha\text{~~and~~}1<\frac{2^h}
{p^g}<\alpha.$$
We claim that $n_p(\alpha)\le 2^{e+1}p^g.$ In order to see this observe
that any integer $n:=2^k p^\ell\ge 2^{e+1}p^g$ satisfies either
$k\ge e+1$, or $\ell\ge g.$ In case 
$k\ge e+1,$ we note that the
number $2^{k-e}p^{\ell+f}$ is even and lies in $[n,n\alpha).$
In case $\ell\ge g,$ we have $2^{k+h}p^{\ell-g}\in [n,n\alpha).$ Next one tries to to find
an even integer 
$n_{\text{new}}:=2^k p^\ell\in [[2^{e+1}p^g/\alpha],2^{e+1}p^g)$, where
$[x]$ denotes the entier of $x$. If successful, we continue until we fail, each time considering the interval
$[[n_{\text{new}}/\alpha],n_{\text{new}})$. 
\begin{exam}
We determine $n_7(5/3)$. 
Starting from $64\cdot 7=448$ we can make either the substitution $32\to 49$ or 
$7\to 8$ with ratio $<5/3$.
Going down from $448$
via $392,256,196,128,98,64,56$, we obtain $n_7(5/3)=34$.
  From $32$ we can go down in several steps to $2$. Thus the integers $n\ge 1$ for which $[n,5n/3)$ does not contain an even number of the form $2^a\cdot 7^b$, are precisely $n=1$ and $n=33$. 
\end{exam}
For some further examples see Tab.\,\ref{tab:125} and \ref{tab:np}. The very large values 
appearing there were determined using more sophisticated techniques involving continued fractions, see Languasco et al.\,\cite{LLMT}.
\begin{table}[h]
\begin{tabular}{|c|c|c|}
\hline
$p$ & $n_p(3/2)$ & $n_p(5/3)$ \\
\hline\hline
$3$ & $2$ & $2$  \\
\hline
$5$ & $22$ & $2$  \\
\hline
$7$ &  $262$ & $34$  \\
\hline
$11$ & $11$ & $10$  \\
\hline
$13$ & $139$ & $10$  \\
\hline
$17$ & $1398102$ & $78644$  \\
\hline
$19$ & $342$ & $308$  \\
\hline
$23$ & $22$ & $20$  \\
\hline
$137$ & $45812984491$ & $2516583$  \\
\hline
$149$ & $21846$ & $19661$  \\
\hline
$271$ & $375299968947542$ & $5153960756$  \\
\hline
\end{tabular}
    \medskip \medskip
\caption{Some values of
$n_p(3/2)$ and $n_p(5/3)$}
\label{tab:np}
\end{table}

\begin{lem}
\label{twointervals}~
\begin{enumerate}[{\rm a)}]
 \item For $n\ge 2^7\cdot 5^6 \, (=2\cdot 10^6)$ the interval $[\frac{380}{453}n,n]$ contains an even integer of
 the form $2^a\cdot 5^b$.
   \item For $n\ge 2^7\cdot 5^{15} \,(=3.90625\cdot 10^{13})$ the interval $[\frac{35}{39}n,n]$ contains an even integer of 
   the form $2^a\cdot 5^b$.
\end{enumerate}
\end{lem}
\begin{proof}
a) Put $\alpha=\frac{453}{380}$. We start by noticing 
that $1<\frac{2^7}{5^3}<\frac{5^7}{2^{16}}<\alpha$. This shows that by making
either the substitution $2^{16}\to 5^7$ or $5^3\to 2^7$, we can
increase the even number $n=2\cdot 2^{16}\cdot 5^2$ in such a way to a further
number of the same format with ratio in $(1,\alpha)$. The so produced sequence of integers is unbounded.
The string of consecutive integers $2m$ with $m=2^6\cdot 5^6,2^{20},2^4\cdot 5^7,2^{18}\cdot 5,2^2\cdot 5^8,2^{16}\cdot 5^2$ also have the property that the ratio of consecutive terms is
in $(1,\alpha)$.

\noindent b) Put $\beta=\frac{39}{35}$. We start by noticing 
that $1<\frac{2^7}{5^3}<\frac{5^{16}}{2^{37}}<\beta$. 
This shows that by making
either the substitution $2^{37}\to 5^{16}$ or $5^{3}\to 2^{7}$, we can
increase the even number $n=2\cdot 2^{37}\cdot 5^2$ in such a way to a further
number of the same format with ratio in $(1,\beta)$. As in the proof of part
a we can lower $n$ to obtain the indicated starting value.
\end{proof}

\begin{prop}
\label{prop:nowildpowersgeneral}
Let $p\ge 13$ be a prime.
If $p^e$ be a potential wild prime power,
then there exist 
integers $a\ge 1$
and $b\ge 0$ such that
\begin{equation}
\label{5over6}
\frac{5(p+1)}{6p}p^e< 2^a\cdot 5^b<p^e.
\end{equation}
\end{prop}
\begin{proof}
For the purposes of this proof we say that $p^e$ is
\emph{approachable} if there are integers $a\ge 1$ and $b\ge 0$ for 
which \eqref{5over6} holds.
We put $\alpha_p=5(p+1)/(6p)$. Note that
$\alpha_{13}=\frac{35}{39}$ and $\alpha_p<\alpha_{151}=\frac{380}{453}$ for $p>151$.

We first suppose that $p\ge 151$. We have
$(\frac{380}{435}p^e,p^e)\subseteq (\alpha_p p^e,p^e)$ and hence $p^e$ is
approachable by Lemma \ref{twointervals}a if $p^e > 2^7\cdot 5^6$. If $p^e<2^7\cdot 5^6$, we
conclude that $151\le p\le 1409$ as $e\ge 2$ by Lemma \ref{neverwild}. 
This leaves only one potential wild prime power, namely $181^2$, 
which turns out to be approachable by $\{2 \cdot 5^6$, $2^8 \cdot 5^3\}$.
\par It remains to deal with the primes $13\le p\le 151$. We have
$(\frac{35}{39}p^e,p^e)\subseteq (\alpha_p p^e,p^e)$ and hence $p^e$ is
approachable by Lemma \ref{twointervals}b if $p^e>2^7\cdot 5^{15}$. This leaves the five potential wild prime 
powers $\{13^7,19^4,19^8,43^7,97^5$\}. These are approachable by $2^5 \cdot 5^9$, $2^3 \cdot 5^6$, $2^6 \cdot 5^{12}$, $2^3 \cdot 5^{15}$, $\{2^5 \cdot 5^{12}, 2^{12} \cdot 5^9, 2^{19} \cdot 5^6, 2^{26} \cdot 5^3\}$, respectively.
\end{proof}
 
\section{Corrections to part I}
\label{proofcorrected}
In part I the conditions on $k$ involving $6\pmod*{9}$ in
the definition of $\cB_{k}$ were 
erroneously omitted. However, the proofs are
only based on the definition 
$\cB_k=\{m~\text{even}:z(m)=m\}.$
Using Lemma 14 (Lemma \ref{appearance} above),
$\cB_k$ was not quite correctly made explicit. 
The upshot is that if one replaces the definition of
$\cB_k$ in part I by the one used here, as far
as we are aware only one
further mathematical correction to part I is needed\footnote{
We use the amended definition for 
results quoted from part I involving 
$\cB_k$.}.
\par In the proof of Lemma 9 around line 12 at p.\,61 it is implicitly 
assumed that $z(p_1^{b_1})=p_1^{b_1-1}(p_1+1)/2,$ 
which is not always guaranteed by our assumption that $z(p_1)=(p_1+1)/2.$
However, by replacing the two lines there by the following ones, 
the proof is effortlessly fixed.
``As for the divisibility by $p_1^{b_1}$, note that since $z(p_1^{b_1})\mid (i-j)$ and $i-j$ is even, it follows that
$$i-j=2z(p_1^{b_1})\ell,$$ for some positive integer $\ell$. Since $\alpha^{z(p_1^{b_1})}\equiv -1\pmod*{p_1^{b_1}}$, it follows that $\alpha^{i-j}\equiv 1\pmod*{p_1^{b_1}}$."
\par We finish this section by pointing out some typos in part I:

\noindent p.\,56, l.\,3. For ``$b\ge 1$" read ``$b=1$".

\noindent p.\,56, l.\,-10. For ``$\Delta(1)=8$" read ``$\Delta(1)=32$".

\noindent p.\,62, l.\,2.
Replace by ``$19m/37\ge z(m)=2^a\cdot p^{b-1}(p+1)/(2k)\ge n.$"

\noindent p.\,63, l.\,-9.
For ``$5/6<2^{a-\alpha-1}<1$" read 
``$5/6<2^{a-\alpha-1}\cdot 5^b<1$".

\noindent p.\,65, l.\,10. The number field $\mathbb K$ is not
defined. It is $\mathbb Q(\sqrt{k(k+1)}).$

\noindent p.\,65. Lemma 14. One should read ``sup" instead of ``lim sup".

\noindent p.\,70, l.\,6. For 
``$p^{a}m_1=m1$" read ``$p^{a}m_1=m$".

\section{Wild prime powers}
Motivated by Corollary \ref{cor:wppmotivation} we make the
following definition.
\begin{defi}[wild prime power]
A prime power $p^e$ with $p\nmid k(k+1)$ such that $p^e$ exactly divides $\cD_k(n)$ for some integer $n,$
we call a wild prime power for $k$.
\end{defi}
Obviously any wild prime power is odd. 
By Corollary \ref{cor:wppmotivation} any discriminator value is divisible by at most one wild prime power.
\begin{lemma}
\label{neverwild}
A prime number is never wild. 
\end{lemma}
\begin{proof}
Suppose that $p$ is a wild prime.
Then $p>2$ and $\cD_k(n)=ap$
for some $k,n$ and an integer $a$ coprime to $p$ 
satisfying $z(a)=a$. In addition, we have
$z(p)=(p+1)/2$. It follows that
$n\le z(ap)\le z(a)z(p)\le a(p+1)/2$. 
Clearly there is a power $2^b$ in
the interval $[(p+1)/2,p).$ 
As $a2^b\ge a(p+1)/2$ is even and satisfies
$z(a2^b)=a2^b$, it
discriminates $U_0(k),\ldots,U_{n-1}(k)$. 
Since $a2^b<ap$, this contradicts the minimality of
$\cD_k(n)$.
\end{proof}
Next we study when $p^e$ with $e>1$ is
wild. This involves the exponent set
$\cM_p$.
\begin{defi}[Exponent set]
Given any odd prime $p$, the exponent set is defined as
$$
\cM_p:=\left\{e\ge 1: \left\{e \frac{\log p}{\log 2}\right\}>1-\frac{\log(1+1/p)}{\log 2}\right\}.
$$
\end{defi}
A simple application of
Weyl's criterion (cf.\,the proof of
\cite[Proposition 1]{PA} or \cite[Proposition 1]{CM}) gives
$$\lim_{x\rightarrow \infty}\frac{\#\{m\in \cM_p:m\le x\}}{x}=\frac{\log(1+1/p)}{\log 2}.$$
In particular, $\cM_p$ is an infinite set.
\begin{prop}
\label{allowedexponent}
Let $p\ge 3$ be a prime. 
The set of integers $e\ge 1$ for which there is no integer
$a$ such that
\begin{equation}
\label{weetniet}
p^e\frac{(p+1)}{2p}\le 2^a<p^e
\end{equation}
equals 
$\cM_p$.
\end{prop}
\begin{proof} 
Put $\rho=(p+1)/(2p)$ and assume 
\eqref{weetniet} does hold for some integers
$a$ and $e$.
By taking logarithms and after some
easy manipulations
\eqref{weetniet} is seen to be equivalent with
$$\frac{\log\rho}{\log 2}\le a-e\frac{\log p}{\log 2}<0.$$
It follows that
$a=\lfloor e\log p/\log 2\rfloor$, 
and we are left with 
$$\left\{e\frac{\log p}{\log 2}\right\}\le -\frac{\log\rho}{\log 2}=1-\frac{\log(1+1/p)}{\log 2}.$$
Hence $\cM_p$ is precisely the
set of $e$ for which \eqref{weetniet} has
no solution.
\end{proof}
Note that the lower bound 
in \eqref{weetniet} is assumed if and only
if $e=1$ and $p$ is a Mersenne prime. 
\begin{table}[h]
\begin{tabular}{|c|c|}
\hline
$p$ & exponent   \\
\hline\hline
$13$ & $7,17,27,37$ \\
\hline
$17$ & $11,22,34$ \\
\hline
$37$ & $19$ \\
\hline
$73$ & $21$ \\
\hline
$97$ & $5,10,15,20,25$ \\
\hline
$181$ & $2,4,\ldots,20,22$ \\
\hline
$1933$ & $12$ \\
\hline
$2389$ &$9$\\
\hline
$4993$ &$7$\\
\hline
$10321$ &$3$\\
\hline
$11290229$ &$7$\\
\hline
    \end{tabular}
    \medskip \medskip
\caption{Potential wild prime powers 
$p^e\le 10^{50}$ with $5<p \le 10^{10}$ 
and $p\equiv 1\pmod*{4}$}
\label{exponent1mod4}
\end{table}

\begin{table}[h]
\begin{tabular}{|c|c|}
\hline
$p$ & exponent   \\
\hline\hline
$3$ & $3,5,8,10,13,15,17,20,22,25,\ldots, 97,99,102,104$ \\
\hline
$7$ & $6,11,16,21,26,32,37,42,47,52,58$ \\
\hline
$11$ & $2,13,15,26,37,39$ \\
\hline
$19$ & $4,8,12,16,20,24,28,32$ \\
\hline
$23$ & $19,21$ \\
\hline
$31$ & $22$ \\
\hline
$43$ & $7,9,14,18,27$ \\
\hline
$67$ & $15$ \\
\hline
$71$ &$20$\\
\hline
$79$ &$23$\\
\hline
$49667$ &$5,10$\\
\hline
    \end{tabular}
    \medskip \medskip
\caption{Potential wild prime powers 
$p^e\le 10^{50}$ with $p \le 10^{7}$ 
and $p\equiv 3\pmod*{4}$}
\label{exponent3mod4}
\end{table}
\begin{defi}[potentially wild prime power]
\label{defi:pot}
A prime power $p^e$ with $e\in \cM_p$ is said to be \emph{potentially wild}.
\end{defi}
Tables \ref{exponent1mod4} and \ref{exponent3mod4} list
smallest potentially wild prime powers.
\begin{lem}
\label{wildtopotentiallywild}
If $p^e$ is a wild prime power, then it
is also potentially wild.
\end{lem}
\begin{proof}
Suppose that $p^e$ is a wild 
prime power for $k$.
Then $\cD_k(n)=p^em$, with 
$p\nmid k(k+1)$ and $m$ composed of only prime factors dividing $k(k+1)$. We assume that
$e\not\in \cM_p$ and derive a contradiction.
We must have $z(p^e)=p^{e-1}(p+1)/2$ and
$z(m)=m$. It follows
that $n\le z(p^em)\le p^{e-1}(p+1)m/2$. Since by assumption $e\not\in \cM_p$, there exists an integer $a$ such that \eqref{weetniet} is satisfied. The number $2^am$ is even, and by Lemma \ref{15+16} we have $\iota_k(2^am)=2^am$. 
Since $2^am\ge p^{e-1}(p+1)/2m\ge n$
we conclude that the numbers
$U_0(k),\ldots,U_{n-1}(k)$ are pairwise distinct modulo $2^am$. Now since $2^am<p^em$, it follows that $p^em$ cannot be a discriminator value, 
a contradiction showing that
$e\in \cM_p$ and hence the potential wildness of $p^e$.
\end{proof}
\begin{rem}
An alternative proof of Lemma 
\ref{neverwild} is obtained on noting that $1\not\in \cM_p$. Thus $p$ is not potentially wild and so not wild.
\end{rem}

The following result represents an important
milepost on our way towards a proof of Theorem \ref{main2}.
\begin{theorem}
\label{pequals5}
Suppose that $k\equiv 1\bmod{3}$.
If $z_k(25)=15$ and $p^e$ is a wild prime power for $k$, then
$p=5$ and $e$ is in $\cM_5$.
\end{theorem}
\begin{proof}
The conditions on $k$ ensure that $(k(k+1),15)=1$. Let $m=\cD_k(n)$ be a discriminator value.
Suppose that a wild prime power
$p^e$ with $p\nmid k(k+1)$ occurs in $m$. 
Then $m=p^e\cdot m_1$ with $z(m_1)=m_1$.
Since we must have
$z(m)>m/2$, it follows that $z(p^e)=p^{e-1}(p+1)/2$. 
The number $m$ discriminates the numbers
$n$ up to at most $(p+1)p^{e-1}m_1/2$. 
\par We want to show that 
$p=5$, and will assume that $p\ne 5$.
Recall that $e\ge 2$.
We first
consider the case where $p^e\in \{3^2,7^2,11^2\}$.
Since $\iota_k(m)$ only depends on the congruence class of $k$ modulo $m$, it 
is a finite computation to verify that
$\iota_k(3^2)\le 4$,
$\iota_k(7^2)\le 7\cdot 3<7^2/2$ if $7\nmid k(k+1)$ and $\iota_k(11^2)\le 11\cdot 5<11^2/2$ if $11\nmid k(k+1)$. By Corollary \ref{incongruence3mod4} we infer from this that, under the above assumptions on $k$, 
$\iota_k(p^e)<p^e/2$ for 
$p\in \{3,7,11\}$. 
We note that $m_1$ is
odd in this three cases as
$p\equiv 3\pmod*{4}$ and hence
$z(m)=\lcm(z(p^e),z(m_1))\le m/2$ otherwise. We can thus
apply the final assertion of Corollary \ref{incongruence3mod4} to conclude that $\iota_{k}(m)<m/2$, which shows that our assumption that $m$ is a discriminator value was wrong to begin with. Thus $p\ge 13$ and so by Proposition 
\ref{prop:nowildpowersgeneral}
there exist integers $a\ge 1$ and 
$b\ge 0$ such that
\begin{equation}
\label{eq:importantinterval}    
\frac{5(p+1)}{6p}p^e<2^a\cdot 5^b<p^e.
\end{equation}
We write $m_1=2^c\cdot m_2$, with $m_2$
odd. 
We now
infer that 
$$z(2^a\cdot 5^b\cdot m_1)=\lcm(z(2^{a+c}),z(5^b),z(m_2))
=\lcm(2^{a+c},3\cdot 5^{b-1},m_2)=3m/5,$$
where we used that 
$(30,m_2)=1$ and
$z(5^b)=5^{b-1}\cdot 3$, which is a consequence of 
$z_k(25)=15$. 
Then $2^a\cdot 5^b\cdot m_1<p^e\cdot m_1$
discriminates the integers up to $3\cdot 2^a\cdot 5^{b-1}\cdot m_1$.
The lower bound part of inequality 
\eqref{eq:importantinterval} now guarantees that $2^a\cdot 5^{b-1}\cdot m_1>(p+1)p^{e-1}m_1/2$, showing that 
$2^a\cdot 5^b\cdot m_1$ is a better discriminator than
$m$. We
conclude that $p=5$. 
By Lemma \ref{wildtopotentiallywild} it follows that
$5^e$ is potentially wild and hence $e\in \cM_5$ by Definition \ref{defi:pot}. 
\end{proof}
\begin{cor}
\label{wildfive}
Suppose that $5^e$ is a wild prime power for an integer $k$
satisfying $k\equiv 1\pmod*{3}$.
Then 
$$e\in \cM_5=\{3,6,9,12,15,18,21,\ldots\}
\quad\text{and}\quad k \equiv 1, 3, 8, 11, 13, 16, 21, 23\pmod*{25}.$$
\end{cor}
\begin{rem}
If we would restrict to wild prime powers of \emph{even} discriminator values, then necessarily $p\equiv 1\pmod*{4}$
and there is no need to consider 
the primes $3,7$ and $11$ separately. In this case
Corollary \ref{incongruence3mod4} is not needed.
\end{rem}
\begin{rem}
Our proof of Theorem 
\ref{pequals5} eventually depends on 
quite a number of numerical coincidences and we are doubtful whether there exists a more conceptual proof. 
\end{rem}
\subsection{Some specific cases}
In this section we will demonstrate Corollary \ref{wildfive}. 
If $m$ has an odd prime divisor, we denote the smallest such by $P_{\text{odd}}(m)$, otherwise we
put $P_{\text{odd}}(m)=1$.
\begin{prop}
\label{actualthere}
\hfil\break
{\rm a)} Suppose that $k\equiv 3\pmod*{5}$ and 
$k\not\equiv 18\pmod*{25}$.
Then
$\cD_{k}(3\cdot 5^{e-1})=5^e$ for every $e$ in $\cM_5$  with $5^e<P_{\text{odd}}(k(k+1))$.\hfil\break
{\rm b)} Suppose that $k\equiv 1,3\pmod*{5}$ and 
$k\not\equiv 6,18\pmod*{25}$.
Then
$\cD_{k}(6\cdot 5^{e-1})=2\cdot 5^e$ for every $e$ in $\cM_5$  with $2\cdot 5^e<P_{\text{odd}}(k(k+1))$.
\end{prop}
\begin{proof}
We only prove part a, the proof of b being similar.
By Lemma \ref{iota5b} we have $\cD_{k}(3\cdot 5^{e-1})\le 5^e$.
The assumption on 
$P_{\text{odd}}$ ensures that up to $5^e$ only powers
of two occur in $\cA_k\cup \cB_k$. It then follows by Theorem \ref{main2}
that $\cD_{k}(3\cdot 5^{e-1})=2^a\cdot 5^b$, with $a,b\ge 0$ and
$2^a\cdot 5^b\le 5^e$. If $b=0$, then we must have
$3\cdot 5^{e-1}\le 2^a<5^e$, contradicting our assumption that
$e\in \cM_5$. We have 
$\iota_k(2^a\cdot 5^b)\le 3\cdot 2^a\cdot 5^{b-1}\le z_k(2^a\cdot 5^b)$.
We require that $3\cdot 2^{a}\cdot 5^{b-1}\ge 3\cdot 5^{e-1}$.
In combination with $2^a\cdot 5^b\le 5^e$, this gives $2^a\cdot 5^b=5^e$,
completing the proof.
\end{proof}
Two simple ways to obtain a $k$ with 
$P_{\text{odd}}(k(k+1))$ large  
are to take $k$ to be a 
power of two such that $k+1$ is a prime, or to take $k$ a prime
such that $k+1$ is a power of two. This 
then leads to the
\emph{Fermat}, respectively \emph{Mersenne primes}. Conjecturally there are only finitely many Fermat primes, but infinitely many Mersenne primes. The largest known Fermat primes is $65537$, in contrast huge Mersenne primes are known. We will thus restrict to the case where $k$ is a Mersenne prime. Proposition \ref{actualthere} then 
has the following corollary.
\begin{cor}
Suppose that $p\equiv 1 \pmod*{4}$ is a prime $>5$ such that
$q:=2^p-1$ is also a prime. Then
$\cD_{q}(6\cdot 5^{e-1})=2\cdot 5^e$,
for all those $e$ in $\cM_5$  for which $2\cdot 5^e<q$. 
\end{cor}
\begin{proof}
This follows from part b, on noting that if $p\equiv 1\pmod*{4}$ and $p>5$,
then $2^p-1\equiv 1\pmod*{5}$ and 
$2^p-1\not\equiv 6\pmod*{25}$.
\end{proof}
We note that $2^{82589933}-1$, the largest known prime number as of Oct.\,2022,  satisfies the conditions of the corollary.
A similar corollary of part a is not possible as any
Mersenne prime $>3$ is $\not\equiv 3\pmod*{5}$.
\par For $n$ large enough the behavior of $\cD_q(n)$ with
$q$ a Mersenne prime is particularly easy as the
following corollary of 
Theorems \ref{oldmaincorrected} and \ref{main2} shows.
\begin{cor}
Let $q=2^p-1>3$ be a Mersenne prime. Then for
$n\ge n_q^o(3/2)$ we have $$\cD_q(n)=\min\{m\ge n:m=2^a\cdot q^b\text{~and~}a,b\ge 0\}.$$
Equality already
holds for $n\ge n_q^o(5/3)$ if $p\equiv 1\pmod*{4}$ and 
$p>5$.
\end{cor}
One finds that $n_7^{o}(3/2)=131$ and some computation leads, for every $n\ge 1$, to  
$$\cD_7(n)=\min\{m\ge n:m=2^a\cdot 7^b\text{~and~}a,b\ge 0\}.$$
For Mersenne primes $q>7$ the numbers
$n_q^o(3/2)$ and $n_q^o(5/3)$  seem to be 
huge \cite{LLMT}, and hence a complete
characterization infeasible.
For example,
$n_{131071}^o(5/3) 
=
\lceil 2^{16} \cdot 131071^{23897}\cdot 3/5 \rceil
$,
a number having $122\,298$ digits!!

\section{Proofs of Theorems \ref{main2} and \ref{thm2refined}}
\label{sec:mainproof}
We first prove Theorem 
\ref{thm2refined}, since it will be used in our proof of Theorem \ref{main2}.
\begin{proof}[Proof of Theorem \ref{thm2refined}]
Recall that $\cS_{k,n}:= \{m \in \cA_k \cup \cB_k : m \ge n\}.$ 
Let $m=\cD_k(n)$ be a discriminator value with $m\not\in \cS_{k,n}$.
If $z(m)=m$, then $m\in \cS_{k,n}$ by \eqref{Dkn=n}and so
we may assume that $z(m)<m$.

By Lemma \ref{neverwild} it follows that $m\in \cM$, with $\cM$ as in Lemma \ref{appearance}.
This entails by Lemma \ref{appearance} that $n\le z(m)\le \sigma_km$ and hence 
$m\ge n/\sigma_k$.
By assumption there is an integer $m_1\in \cS_{k,n}$ with $m<n/\sigma_k$, showing that
$m_1$ is a better discriminator than $m$. We conclude that 
$m\in \cS_{k,n}$. In case $m$ is even, we repeat this proof, but this time using the inequality $z(m)\le \tau_k m$.
\end{proof}

\begin{proof}[Proof of Theorem \ref{main2}]
By Theorem \ref{1+2} we may assume that
$k>1$.
Suppose that $k\not\equiv 1\pmod*{3}$. 
Thus $6 \mid k(k+1)$, and so all integers of the form $2^a\cdot 3^b$ 
with $a \ge 1$ and $b \in \{0, 1\}$ belong to $\cB_k$. 
We have $\sigma_k\le 3/5$. 
Now by  Theorem \ref{thm2refined} it
suffices to check that 
for every $n \ge 2$ there is an integer $m_1$ of the form $2^a\cdot 3^b$  with $a \ge 1$ and $b \in \{0,1\}$ in
the interval $[n,5n/3)\subseteq [n,n/\sigma_k)$. This can be done 
with help of the substitutions $3 \to 2^2$ and $2 \to 3$, 
which can be applied starting from $n=2$.
\par Next suppose that $k\equiv 1\pmod*{3}$ and $k \equiv 0, 4\pmod*{5}$. Now $10 \mid k(k+1),$ and so all 
integers of the form $2^a\cdot 5^b$ with  $a \ge 1$ and $b \ge 0$ belong to $\cB_k$.
We apply Theorem \ref{oldmaincorrected}.
It is easy to see that for every $n\ge 22$ the interval $[n,2n/3)$ contains
an integer of the form $2^a\cdot 5^b$ with $a\ge 1$ and $b\ge 0$.
Thus if $m\not\in \cS_{k,n}$, then $n\le 21$.
The odd primes $\le 21$ are not wild
by Lemma \ref{wildtopotentiallywild}.
This leaves us only with $9$. However, $9$ is not 
potentially wild (see Tab.\,\ref{exponent3mod4}) and so certainly not wild.
\par Put $m=\cD_k(n)$. Either $z(m)=m$ or $z(m)<m$, in which case we
have $m=p^e\cdot m_1$ with $z(m_1)=m_1$ and $p^e$ a wild prime power.
By Theorem \ref{pequals5} we conclude that $p=5$. It follows that either
$z(m)=m$ or $m=5^bm_1$ with $z(m)=3m/5$ and $z(m_1)=m_1$. In part I we
established that the only discriminator
values $m$ with $z(m)=m$ satisfy $m\in \cA_k\cup \cB_k$.
These discriminate $U_0(k),\ldots,U_{m-1}(k)$
and hence $\cD_k(n)\le \min\{\cS_{k,n}\}$. It remains to deal with the case where
$m=5^b\cdot m_1$ with $b\ge 1$, $z(m)=3m/5$ and $z(m_1)=m_1$.
On invoking Lemma \ref{laatsteloodje} the
proof of \eqref{mainres} is now completed.

\par Let $p>2$ be a prime divisor of 
$k(k+1)$.
If $n\ge n_p(5/3)$, then the interval
$[n,5n/3)$ contains an even integer of the 
form $2^a\cdot p^b$. This number is
in $\cB_k$ and $\ge n$ and so in $\cS_{k,n}$.
As it is less than $5n/3$, it follows by \eqref{mainres} that
$\cD_k(n)=\min \cS_{k,n}$.
\par By assumption $k$ has an odd 
prime divisor $p$. If $n\ge n_p^o(5/3)$, then the interval
$[n,5n/3)$ contains an integer of the 
form $2^a\cdot p^b$. This number is
in $\cA_k\cup \cB_k$ and $\ge n$ and so in $\cS_{k,n}$.
As it is less than $5n/3$, it follows by \eqref{mainres} that
$\cD_k(n)=\min \cS_{k,n}$.
\end{proof}

\section{Effective bounds 
for wild prime powers and elements in $\cF_k$} 
\label{effective}
It was proved in part I that the set $\cF_k$ is
finite for $k>1$. Here, we precise our proof by showing that this set
can be effectively determined and establish Theorem \ref{largefkbound}.
\begin{definition}[prime types]
We say that a prime $p$ is of
\begin{itemize}
\item type I\phantom{V} if $p\mid k$;
\item type II\phantom{I} if $p\mid k+1$;
\item type III if $p\nmid k(k+1)$, $e_p(k)=1$;
\item type IV if $p\nmid k(k+1)$, $e_p(k)=-1$,
\end{itemize}
with $e_p(k)$ as defined in 
\eqref{epk}.
\end{definition}

\begin{lemma}
\label{finitelymanydiscs}
Let $k\ge 2$.
There are only finitely many odd discriminators which are not made up of
primes $p$ dividing $k$.
\end{lemma}

\begin{proof}
Let $m$ be an odd discriminator value not made up only of primes of type
I. Then by Corollary \ref{cor:wppmotivation} we can write $m=p_1^{a_1}\cdot m_1$,
where $p_1$ is of type IV and unique, and $m_1$ is only made up of primes of type
I or II. We will show that both $p_1^{a_1}$ and $m_1$ are bounded. Note that
$$
z(m)={\text{\rm lcm}}[z(p_1^{a_1}),z(m_1)].
$$
We may assume  that $z(p_1^{a_1})=p_1^{a_1-1}(p_1+1)/2,$ for otherwise
$z(m)<m/2$, contradicting our assumption that $m$ is a discriminator value. 
If there is a power of $2$, say $2^b$, in the 
the interval $[p_1^{a_1-1}(p_1+1)/2,
p_1^{a_1}]$, 
then $2^b\cdot m_1<m$ is a better discriminator than $m$.
We thus may assume there is no power of $2$ in this
interval, which guarantees the existence of
an integer $a$ such that
$$
p_1^{a_1}< 2^{a+1}<p_1^{a_1}\left(1+1/p_1\right).
$$
Thus, $p_1^{a_1}>(p_1/(p_1+1))2^{a+1}$. Since $p_1\ge 3$, it follows
that $p_1^{a_1}>(3/4)2^{a+1}.$ Further,
$$
z(p_1^{a_1})=p_1^{a_1}\left(\frac{p_1+1}{2p_1}\right)\le
  \frac{2p_1^{a_1}}3 < \frac{2^{a+2}}3.
$$
Now let $p$ be any odd prime factor dividing $k(k+1)$. 
Since $k(k+1)$ cannot be a power of $2$ for $k>1$ 
such a prime $p$ exists. We search for
a pair of positive integers $(u,v)$ such that
$$
\frac{2}{3}\cdot 2^{u+1}<p^v<\frac{3}{4}\cdot 2^{u+1}.
$$
This we find quickly, since the above condition is equivalent to
\begin{equation}
\label{eq:u}
\left\{(u+1)\frac{\log 2}{\log p}\right\}\in
\left(\frac{\log(4/3)}{\log p},\frac{\log(3/2)}{\log p}\right),
\end{equation}
and the sequence of fractional parts $\{nx\}$ is dense (even uniformly
distributed) for irrational $x$. Let $u$ be the minimal positive 
integer with this
property. 
Note that the corresponding $v$ is uniquely determined.
By contradiction we will now show that
$a\le u$, leading to the bound
\begin{equation}
\label{eq:2}    
\ell:=p_1^{a_1}<2^{u+1}.
\end{equation}
Assume that $a>u$ is any integer. 
We note that
$$p_1^{a_1-1}(p_1+1)/2<(2/3)2^{a+1}<2^{a-u}\cdot p^v<(3/4)2^{a+1}<p_1^{a_1}.$$
Thus,
$$
2^{a}<\frac{p_1^{a_1-1}(p_1+1)}{2}<2^{a-u} 
\cdot p^v<p_1^{a_1}<2^{a+1},
$$
and we conclude that $m_2:=m_1 2^{a-u}p^v$
has the property that $z(m_2)=m_2$. 
If $n$ satisfies $\cD_k(n)=m$, then
$n\le m_1\cdot z(p_1^{a_1})$.
The even integer $m_2$ 
satisfies $m_1\cdot z(p_1^{a_1})<m_2<m=m_1\cdot p_1^{a_1}$ and
discriminates the integers $U_0(k),\ldots,U_{m_1z(p_1^{a_1})-1}(k)$, contradicting 
the (discriminatory) minimality of $m$.
This shows that, if
$p_1^{a_1}$ is such that $p_1^{a_1}<2^{a+1}<p_1^{a_1}+p_1^{a_1-1}$ and
$m$ is actually a discriminator, then $a\le u$ and
\eqref{eq:2} is satisfied. 

Fix $p_1^{a_1}=\ell$ and let $t=z(\ell)<\ell$. Now we look at the
numbers $m_1 \cdot t<m_1\cdot \ell$ with $m_1>1$. Let $q$ be any odd prime dividing $m_1$. Let
$(e_q,f_q)$ be the first pair of indices such that $q^{e_q}
\cdot t<2^{f_q}<q^{e_q}\cdot \ell$ (it exists because of an argument with fractional parts as above). Then,
if $q^e$ divides $m_1$ with $e\ge e_q$, we can replace $q^e$ by $q^{e-e_q}
\cdot 2^{f_q}$. This has the effect of replacing $m_1$ by $m_1 \cdot 2^{f_q}\cdot q^{-e_q},$
which is a better discriminator for the numbers $n\le m_1\cdot t$ than the number $m_1\cdot \ell$ is.
This can be done for each $q$ dividing $m_1$. Since there are only finitely
many $q$ (namely, odd primes of type I and II), we see that $m$ is
bounded.
\end{proof}

To make the argument effective we need to find
$N$ so that the containment condition \eqref{eq:u} holds  
for some positive
integer $u\le N$ and bound $\ell$ in \eqref{eq:2}.
\par Let $\theta=\log 2/\log p$. Note that $\theta\not\in \Q$.  Recall now
that the \textit{discrepancy} $D_N$ of a sequence $\{a_m\}_{m=1}^N$
of real numbers (not necessarily distinct) is defined as
$$
D_N=\sup_{0\le \gamma\le 1}\left|\frac{\#\{m\le N~:~\{a_m\}<\gamma\}}{N}-
\gamma\right|.
$$
  From the above definition we see that the inequality
$$
\#\{m\le N~:~\alpha\le \{a_m\}<\beta\}\ge (\beta-\alpha)N-2D_N N
$$
holds for all $0\le \alpha\le \beta\le 1$.
Thus, setting $a_m=m\theta$ for all $m=1,\ldots,N$,  and letting
$$
I=\left(\frac{\log(4/3)}{\log p},\frac{\log(3/2)}{\log p}\right),
$$
which is an interval of length $\log(9/8)/\log p$, we have
\begin{equation}
\label{eq:ineqforN}
\#\{m\le N: \{a_m\}\in I\}  \ge   | I |  N
-2D_N N=   \left(\frac{\log
(9/8)}{\log p}\right)N-2D_N N.
\end{equation}
In particular, if the right-hand side is positive, then there is $u\le
N$ with $\{a_u\}\in I$.
We now upper bound $D_N$. The Koksma-Erd\H os-Tur\'an
inequality (see Lemma 3.2 in~\cite{KuiNied})
bounds the discrepancy $D_N$ by
\begin{equation}
\label{eq:ErdosTuran}
D_N\le \frac{3}{H} + \frac{3}{N}\sum_{m=1}^{H}\frac{1}{m\|a_m\|},
\end{equation}
where $\|x\|$ is the distance from $x$ to the nearest
integer and $H\le N$ is an arbitrary positive integer.
\medskip

To bound $\|a_m\|$, note that
$$
\|a_m\|=\left|m\frac{\log 2}{\log p}-t\right|= \frac{1}{\log
p}|m\log 2-t\log p|,
$$
where $t$ is an integer such that $t\le m(\log 2)/(\log p)+1<2m$. 
Note that $\|a_m\|\ne 0,$ since $\theta\in\R\setminus \Q$. Thus,
$|m\log 2-t\log p| \neq 0$ and a lower bound for it can be
obtained by using the theory of linear forms in logarithms.

Let us recall Matveev's main theorem \cite{Mat}. It applies to algebraic
numbers,
but we
recall it here only for rational numbers.
For a rational number $\gamma=r/s$ with coprime integers $r$ and $s>0$,
let $h(\gamma):=\max\{\log |r|, \log s\}$.

\begin{theorem}[Matveev \cite{Mat}]
\label{Matveev11} Let $\gamma_1,\ldots,\gamma_k$ be positive rational
numbers, let $b_1,\ldots,b_k$ be non-zero integers, and assume that
\begin{equation}
\label{eq:Lambda}
\Lambda:=\gamma_1^{b_1}\cdots\gamma_k^{b_k} - 1,
\end{equation}
is non-zero. Then for every 
$$
B\geq\max\{|b_1|, \ldots, |b_k|\}
$$
we have
$$
\log |\Lambda| > -1.4\cdot 30^{k+3}\cdot k^{4.5}\,(1+\log B)\,
h(\gamma_1)\cdots h(\gamma_k).
$$
\end{theorem}
In our case, we take
$$
\Lambda=2^m\cdot p^{-t}-1,
$$
which is non-zero, since $p$ is an odd prime.  Note that
$\Lambda=e^{\Gamma}-1$, where $\Gamma=m\log 2-t\log p$.
So, either $|\Gamma|\ge 1/2$, or $|\Gamma|<1/2$. If $|\Gamma|<1/2$,
then
$$
2|\Gamma|>|e^{\Gamma}-1|=|\Lambda|,
$$
and we can apply Matveev's theorem to get a lower bound 
on $|\Lambda|$ and hence on $|\Gamma|.$ Either way, we take in Matveev's theorem
$$
k=2,\quad \gamma_1=2,\quad \gamma_2=p,\quad b_1=m,\quad b_2=-t,
$$
and, noting that we can set $B:=2m$, we get
\begin{equation}
\label{eq:Matv}
2|m\log 2-t\log p|  > \exp\left(-C_1(\log 2) (1+\log (2m))  \log p\right),
\end{equation}
where $C_1=1.4\cdot 30^5\cdot 2^{4.5}$. Since
$
1.4\cdot 30^5\cdot 2^{4.5}\cdot \log 2<6\cdot 10^8-\log 2,
$
we get
$$
|m\log 2-t\log p|>\exp(-6\cdot 10^8(1+\log(2m))\log p)=
p^{-6\cdot 10^8
(1+\log (2m))}\quad {\text{\rm for}}\quad m\ge  1.
$$
We thus obtain that, if $H\ge 15$ and $2m\le H,$ then
$$
1+\log (2m)\le 1+\log H\le 1.63\log H\qquad (H\ge 15),
$$
and so the inequality \eqref{eq:Matv} leads to
$$
\frac{1}{\|a_m\|}\le (\log p)\,p^{(6\cdot 1.63) \cdot 10^8 \log
H}<p^{(10^9-2)\log H}=H^{(10^9-2)\log p}<H^{10^9\log p-2}.
$$
Thus,
$$
D_N\le 3\left(\frac{1}{H}+\frac{H^{10^9\log p-2}}{N} \sum_{m=1}^H
\frac{1}{m}
\right)<3\left(\frac{1}{H}+\frac{H^{10^9\log p-1}}{N}\right).
$$
where we trivially bounded the sum by $H$.
Choosing $H:=\left\lfloor N^{10^{-9}/\log p}\right\rfloor$ we get,
assuming
still that $H\ge 15$ and therefore that
\begin{equation}
\label{eq:hyp}
N^{10^{-9}/\log p}\ge 15,\quad {\text{\rm which is equivalent
to}}\quad N\ge 15^{10^9\log p},
\end{equation}
that
$$
D_N\le 3\Big(\frac{1}{H}+\frac{H^{10^9\log p-1}}{N}\Big)\le
3\left(\Big\lfloor N^{10^{-9}/\log
p}\Big\rfloor^{-1}+N^{-10^{-9}/\log p}\right)\le
7N^{-10^{-9}/\log p},
$$
where we use the 
trivial observation that if $x\ge 15$, then
$$
\frac{1}{\lfloor x\rfloor}+\frac{1}{x}\le 
\frac{1}{x}\Big(\frac{1}{1-\frac{1}{x}}+1\Big)\le \frac{29}{14}\cdot\frac{1}{x}<\frac{7}{3}\cdot 
\frac{1}{x}.$$
Turning now our attention to the inequality \eqref{eq:2}, we see that
\begin{equation}
\label{koksma}    
N\left(\frac{\log(9/8)}{\log p}-2D_N\right)>N\left(\frac{\log(9/8)}{\log
p}-14N^{-10^{-9}/\log p}\right).
\end{equation}
Thus, if $N\ge N_0$ with
\begin{equation}
\label{eq:N0}
N_0:= \left(\frac{15\log p}{\log(9/8)}\right)^{10^9\log p},
\end{equation}
the right--hand side of 
\eqref{koksma} is at 
least
\begin{equation}
\label{N15}
\frac{N}{15}\frac{\log(9/8)}{\log p}
\end{equation}
and hence positive.
Note that the inequality $N\ge N_0$, with $N_0$ as in \eqref{eq:N0},
ensures that the inequality \eqref{eq:hyp} is satisfied. 
Hence, we have established the following
result.
\begin{lemma}
\label{lem:10}
Let $p$ be an odd prime factor of $k(k+1)$. There is a positive integer $u$ such that
\begin{equation}
\label{eq:containment}
\left\{(u+1)\frac{\log 2}{\log p}\right\}\in \left(\frac{\log(4/3)}{\log
p},\frac{\log(3/2)}{\log p}\right)
\end{equation}
and 
$$
u+1<\left(\frac{15\log p}{\log(9/8)}\right)^{10^9\log p}.
$$
\end{lemma}

The argument can be iterated to give an upper bound on the largest
element of $\cF_k$.

\begin{lemma}
\label{lem:disc}
If $m$ is an odd discriminator not entirely made up of primes dividing
$k$, then
$$
m<2^{(k+1)^{10^{10}\log\log (k+1)}}.
$$
\end{lemma}

\begin{proof}
We keep the notation from the proof of Lemma \ref{finitelymanydiscs}. As
such, we write $m=p_1^{a_1}\cdot m_1$, where $m_1$ is made up
of primes $p$ dividing $k(k+1)$. 
By the argument from that proof, we conclude that
$$
p_1^{a_1}< 2^{a+1}\le 2^{u+1},
$$
with $u$ the smallest integer satisfying \eqref{eq:u}.
By Lemma \ref{lem:10} and 
since $15/\log(9/8)<130$, it follows that
$$
u+1\le (130\log p)^{10^9\log p}=p^{10^9\log(130\log p)}.
$$
Therefore, with the notation of Lemma \ref{finitelymanydiscs}, we have
\begin{equation}
\label{eq:ell}
\ell=p_1^{a_1}< 2^{u+1}\le 2^{p^{10^9\log(130\log p)}}.
\end{equation}
Now let $q\mid m_1$, which implies $q\mid k(k+1)$. We need to estimate
the smallest pair of positive integers $(e_q,f_q)$ such that, if we put
$$
t:=z(\ell)=p_1^{a_1-1}(p_1+1)/2,
$$ 
then
\begin{equation}
\label{eq:101}
q^{e_q}\cdot t<2^{f_q}<q^{e_q}\cdot \ell.
\end{equation}
Taking logarithms, we have
$$
e_q+\frac{\log \ell}{\log q}-\frac{\log(\ell/t)}{\log q}<f_q\frac{\log
2}{\log q}<e_q+\frac{\log \ell}{\log q}.
$$
Since $\ell/t=2p_1/(p_1+1)\ge 3/2$, the above condition places
$\{f_q(\log 2)/(\log q)\}$ in one (or two) intervals of total length
$\log(\ell/t)/(\log q)\ge \log(3/2)/(\log q)$.
More precisely, if $\{\log\ell/\log q\}>\log(3/2)/(\log q)$, it then
follows that it suffices that
$$
\left\{f_q\frac{\log 2}{\log q}\right\}\in \left(\left\{\frac{\log
\ell}{\log q}\right\}-\frac{\log(3/2)}{\log q}, \left\{\frac{\log
\ell}{\log q}\right\}\right),
$$
whereas if $\left\{\log \ell/\log q\right\}<\log(3/2)/\log q$, it
suffices that
$$
\left\{f_q\frac{\log 2}{\log q}\right\}\in \left(1+\left\{\frac{\log
\ell}{\log q}\right\}-\frac{\log(3/2)}{\log q}, 1\right)\cup
\left(0,\left\{\frac{\log \ell}{\log q}\right\}\right).
$$
In any case, there is an interval $J$ of length $0.5\log(3/2)/\log q$
such that, if $\{f_q\log 2/\log q\}\in J$, then the estimate
\eqref{eq:101} holds with some appropriate positive integer $e_q$.
Note that $0.5\log(3/2)>\log(9/8)$, so by the arguments from the proof
of Lemma \ref{lem:10}, it follows that if $N\ge N_0,$ where $N_0$
satisfies \eqref{eq:N0}, then there are at least
$$
\frac{N\log (9/8)}{15\log p}
$$
values of $f\le N$ such that $\{f\log 2/\log q\}\in J$. This in turn
implies the inequalities $q^et<2^f<q^e\ell$. The only situation in which we are in
trouble is when $e=0,$ in which case $t<2^f<\ell$. Assume this happens. By inequality \eqref{eq:ell}, we get
$$
f<p^{10^9\log(130\log p)}.
$$
If this were so for all the acceptable values for $f$, we would get 
by \eqref{N15} that 
$$\frac{N\log (9/8)}{15\log p}<p^{10^9\log(130\log p)},$$
and therefore
\begin{equation}
\label{eq:ttt}
 N<\left(\frac{15\log p}{\log(9/8)}\right)p^{10^9\log(130\log p)}<p^{(10^9+1)\log(130\log p)}.
\end{equation}
To ensure that this doesn't happen we ask that $N\ge N_1$, where
\begin{equation}
\label{eq:N1}
N_1=\left(\frac{15\log p}{\log(9/8)}\right)^{1.1\cdot 10^9\log p}.
\end{equation}
Indeed, since $15/\log(9/8)>127$, the above inequality forces 
$$
N>p^{1.1\cdot 10^9 \log(127\log p)}.
$$
To see that \eqref{eq:ttt} fails for such $N$, assume it doesn't and we get
$$
p^{1.1\cdot 10^9 \log(127\log p)}<N<p^{(10^9+1)\log(130\log p)},
$$
and so 
$$
\frac{\log(130\log p)}{\log(127\log p)}>\frac{1.1\cdot 10^9}{10^9+1}>1.09.
$$
However, this is false since the function $(\log(130)+x)/(\log(127)+x)$ on the left with  $x=\log\log p$ is decreasing for $x\ge 0$ with the maximum $\log(130)/\log(127)=1.04\ldots$ 
at $x=0$, which is not larger than $1.09$. It then follows that by choosing $N$ as in \eqref{eq:N1} then there is some $f$ such that $q^e\cdot t<2^f<q^e \cdot \ell$ and $e>0$. 
Since $15/\log(9/8)<130$, it follows that, in particular,
$$
f_q\le \left(130\log p\right)^{1.1\cdot 10^9\log p},
$$
and since $p\le k+1$, we get
$$
q^{e_q}\le 2^{f_q}\le 2^{(k+1)^{1.1\cdot 10^9\log(130\log(k+1))}}.
$$
Thus,
\begin{equation}
\label{hugebounds}    
m_1 \le  \prod_{\substack{q\mid k(k+1)\\ q~{\text{\rm odd}}}}
q^{e_q}\le 2^{\omega(k(k+1)) (k+1)^{1.1\cdot
10^9\log(130\log(k+1))}} < 2^{(k+1)^{10^{10}\log\log(k+1)}}.
\end{equation}
The right-most inequality follows because of the trivial estimate
$$
\omega(k(k+1))=\omega(k)+\omega(k+1)\le \frac{2\log(k+1)}{\log
2}<4\log(k+1)\le (k+1)^3,
$$
and, furthermore,
$$
3+1.1\cdot 10^9\log(130\log(k+1))<10^{10}\log\log(k+1),
$$
which holds for $k\ge 6$. One may check by hand that this is also true
for $k\in \{2,3,4,5\}$. Indeed, in these cases $p\in \{3,5\}$, and one checks that in each of the cases one may choose 
$u\le 20$ satisfying the containment \eqref{eq:containment} of Lemma \ref{lem:10}. 
\end{proof}

\begin{lemma}\label{lemma:finitelydiscseven}
Let $k\ge 2$.
There are only finitely many discriminators which are
even  and not
divisible only by primes $p$ dividing $k(k+1)$.
\end{lemma}

\begin{proof}
We write $m=2^z\cdot m_1$ with $z\ge 1$ and 
$m_1$ odd. The previous arguments showed that
$m_1$ has at most one prime factor not of type I or II. If it has one, it
is of type IV. Assume $m_1=p_1^{a_1} \cdot m_2$, where $p_1$ is of type IV.
Then $z(p_1)\mid (p_1+1)/2$. If $z(p_1)\mid (p_1+1)/4$, then $z(m)\le
2^{a} m_2p_1^{a_1-1}(p_1+1)/4<m/2$, and we get a contradiction. A
similar contradiction is obtained if $z(p_1^2)\mid (p_1+1)/2$, so we may
assume that $z(p_1^{a_1})=p_1^{a_1-1}(p_1+1)/2$. As in previous
occasions, there exists $a$ with
$$
p_1^{a_1}<2^{a+1}<p_1^{a_1}(1+1/p_1).
$$
Thus, $p_1^{a_1}>(3/4) 2^{a+1}$. As in the previous application, we pick
an odd prime $q$ dividing $k(k+1)$ and 
let $u$ be
minimal such that
$$
\frac{2}{3}\cdot 2^{u+1}<q^v<\frac{3}{4}\cdot 2^{u+1}$$
for some (unique) $v$.
Then, if $a>u$, it follows that
$2^{a-u} q^v m_2$ is a
better discriminator than $m_1$. This shows that $p_1^{a_1}<2^{a+1}\le
2^{u+1}$ is bounded. The bound is the same as in Lemma
\ref{lem:10}. Next, for each odd prime $q\mid m_2$ (of type I or
II) we find $(e_q,f_q)$ such that $q^{e_q}\cdot t< 2^{f_q}<q^{e_q}\cdot \ell$, with
$(t,\ell)=(z(p_1^{a_1}),p_1^{a_1}),$ for all
finitely many choices $p_1^{a_1}$. Then, if the exponent of $q$ in $m$
exceeds $e_q$, we can replace $m$ by $m\cdot 2^{f_q}\cdot q^{-e_q}$, which yields a
better discriminator. This holds for all prime factors $q$ of $m_2$, so
also $m_2$ is bounded. 
The bounds given in \eqref{hugebounds} apply
to $m_1$.
\par It remains to bound the exponent $z$ of $2$ in the factorization of $m$.
We know by now that $m=2^z m_1$
and that $m_1$ is odd and bounded, 
cf.\,\eqref{hugebounds}, as
$$
m_1<2^{\omega(k(k+1)) (k+1)^{1.1\cdot 10^9\log(130\log(k+1))}}.
$$
Further, $z(m_1)<m_1$. Put $(t,\ell)=(z(m_1),m_1)$. Again, we pick
some odd prime $q$ dividing $k(k+1)$ and search for
integers $x,y$ such that the inequality $2^x\cdot t<q^y<2^x\cdot \ell$ holds, where $\ell=m_1$ and
$t=z(m_1)$. This is equivalent to
$$
x+\frac{\log \ell}{\log 2}-\frac{\log(\ell/t)}{\log 2}<y\frac{\log
q}{\log 2}<x+\frac{\log \ell}{\log 2}.
$$
Note that this is again satisfied if $\{y\log q/\log 2\}$ is in one or
two intervals of length $\log(\ell/t)/\log 2>\log(3/2)/\log 2$. The
argument with linear forms in logarithms works and
gives a bound on $y$ as in Lemma \ref{lem:10}. This shows that
$$
2^x<q^{y}<(k+1)^{(k+1)^{10^9\log(130\log(k+1)}}<2^{(\log(k+1)/\log 2)
(k+1)^{10^9\log(130\log(k+1))}}.
$$
If $x>0$, and $z>x$ then we replace $2^zm_1$ by $2^{z-x} m_1q^y$, which
is a better discriminator. This shows that $z\le x$ when $x$ is positive. To ensure that $x$ is positive we argue as in the previous lemma to conclude that 
if we replace the exponent $10^9\log(130\log(k+1))$ by $1.1\cdot 10^9 \log(130\log(k+1))$, then there is a choice of $(x,y)$ with $x>0$. For such $x$ we have $z\le x$ is also bounded, so
$$
m<2^{(\omega(k(k+1))+(\log(k+1)/\log 2))(k+1)^{1.1\cdot
10^9\log(130\log k)}}<2^{(k+1)^{10^{10}\log\log(k+1)}},
$$
which is what we wanted to show. Again the last inequality holds for $k\ge 6$ and for smaller values of $k$ can be checked by hand. 
\end{proof}

\section*{Acknowledgments}
Work on this article was started during
a February-June 2017 stay of the second author
at the Max Planck Institute for Mathematics (MPIM)
and continued during further stays in 
September 2019-February 2020, June 2021 and a few
days in April and July 2023. (We stress
that these stays were only very partially devoted to work on
this paper.) The second and third author thank the MPIM for making
these stays possible.
Alessandro Languasco provided a lot of 
help with
computing Table \ref{tab:np} and $n_p(\alpha)$
in general (see \cite{LLMT} for a description of his algorithm). This required a big investment of both his time and that of his CPU's.
Thanks are also due to Alexandru Ciolan for
his input in some early versions.

\end{document}